\documentclass[12pt]{amsart}

\usepackage{tgpagella}
\usepackage{euler}
\usepackage[T1]{fontenc}
\usepackage{amsmath, amssymb}
\usepackage[hidelinks]{hyperref}
\usepackage[english]{babel}
\usepackage{mathrsfs}
\usepackage{eucal}
\usepackage[all]{xy}
\usepackage{tikz}

\newtheorem{thm}{Theorem}[section]
\newtheorem*{thm*}{Theorem}
\newtheorem{lem}[thm]{Lemma}
\newtheorem{fact}[thm]{Fact}

\newtheorem{prop}[thm]{Proposition}
\newtheorem*{prop*}{Proposition}

\newtheorem{cor}[thm]{Corollary}
\newtheorem*{cor*}{Corollary}

\theoremstyle{definition}
\newtheorem{defn}[thm]{Definition}
\newtheorem*{defn*}{Definition}

\newtheorem{remark}[thm]{Remark}
\newtheorem{remarks}[thm]{Remarks}
\newtheorem{question}[thm]{Question}

\newtheorem{examples}[thm]{Examples}
\newtheorem{claim}[thm]{Claim}
\newtheorem*{question*}{Question}
\newtheorem*{Pquestion*}{Popa's question}

\newtheorem*{conv*}{Convention}

\newcommand{\dminus}{ 
\buildrel\textstyle\ .\over{\hbox{ 
\vrule height3pt depth0pt width0pt}{\smash-} 
}}

\def\bb{\mathbb}

\def\bb{\mathbb}

\def\cal{\mathcal}

\def\u{\mathsf 1}

\newcommand{\cstar}{$\mathrm{C}^*$}

\makeatletter

\def\dotminussym#1#2{%
  \setbox0=\hbox{$\m@th#1-$}%
  \kern.5\wd0%
  \hbox to 0pt{\hss\hbox{$\m@th#1-$}\hss}%
  \raise.6\ht0\hbox to 0pt{\hss$\m@th#1.$\hss}%
  \kern.5\wd0}

\newcommand{\dpr}{^{\prime\prime}}

\DeclareMathOperator{\alg}{alg}
\DeclareMathOperator{\id}{id}
\DeclareMathOperator{\op}{op}

\DeclareMathOperator{\Ad}{Ad}
\DeclareMathOperator{\Aut}{Aut}
\DeclareMathOperator{\Inn}{Inn}

\DeclareMathOperator{\Tr}{Tr}

\DeclareMathOperator{\Ball}{Ball}
\DeclareMathOperator{\core}{c}

\DeclareMathOperator{\diag}{diag}

\newcommand\bN{{\mathbb N}}
\newcommand\bR{{\mathbb R}}
\newcommand\bZ{{\mathbb Z}}
\newcommand\bC{{\mathbb C}}

\def \Th{\operatorname{Th}}
\def \R{\mathcal R}
\def \u{\mathcal U}
\def \v{\mathcal V}

\allowdisplaybreaks
\makeatletter
\def\l@subsection{\@tocline{2}{0pt}{2.5pc}{5pc}{}}
\def\l@subsubsection{\@tocline{2}{0pt}{5pc}{7.5pc}{}}
\makeatother

\begin{document}


\title{Existentially closed W*-probability spaces}

\author{Isaac Goldbring}
\address{Department of Mathematics\\University of California, Irvine, 340 Rowland Hall (Bldg.\# 400),
Irvine, CA 92697-3875}
\email{isaac@math.uci.edu}
\urladdr{http://www.math.uci.edu/~isaac}
\thanks{Goldbring was partially supported by NSF grant DMS-2054477.}

\author{Cyril Houdayer}
\address{Universit\'e Paris-Saclay, Institut Universitaire de France, CNRS, Laboratoire de math\'ematiques d'Orsay, 91405 Orsay, FRANCE} \email{cyril.houdayer@universite-paris-saclay.fr}
\thanks{Houdayer is supported by Institut Universitaire de France.}
\urladdr{https://cyrilhoudayer.com}

\maketitle

\begin{abstract}
We study several model-theoretic aspects of W$^*$-probability spaces, that is, $\sigma$-finite von Neumann algebras equipped with a faithful normal state.  We first study the existentially closed W$^*$-spaces and prove several structural results about such spaces, including that they are type III$_1$ factors that tensorially absorb the Araki-Woods factor $R_\infty$.  We also study the existentially closed objects in the restricted class of W$^*$-probability spaces with Kirchberg's QWEP property, proving that $R_\infty$ itself is such an existentially closed space in this class.  Our results about existentially closed probability spaces imply that the class of type III$_1$ factors forms a $\forall_2$-axiomatizable class.  We show that for $\lambda\in (0,1)$, the class of III$_\lambda$ factors is not $\forall_2$-axiomatizable but is $\forall_3$-axiomatizable; this latter result uses a version of Keisler's Sandwich theorem adapted to continuous logic.  Finally, we discuss some results around elementary equivalence of III$_\lambda$ factors.  Using a result of Boutonnet, Chifan, and Ioana, we show that, for any $\lambda\in (0,1)$, there is a family of pairwise non-elementarily equivalent III$_\lambda$ factors of size continuum.  While we cannot prove the same result for III$_1$ factors, we show that there are at least three pairwise non-elementarily equivalent III$_1$ factors by showing that the class of full factors is preserved under elementary equivalence.     
\end{abstract}

\tableofcontents
\section{Introduction}

The model-theoretic study of von Neumann algebras began in earnest with the series of papers \cite{FHS13}, \cite{FHS141}, and \cite{FHS142} by Farah, Hart and Sherman.  There, a theory in a (continuous) first-order language was described for which there was an equivalence of categories between the models of the theory and the category of tracial von Neumann algebras for which the model-theoretic ultraproduct construction corresponded to the tracial ultraproduct construction.   In the time since these papers appeared, there has been a very interesting interplay between model-theoretic and operator-algebraic techniques; recent examples of applications of model-theoretic techniques to problems about tracial von Neumann algebras can be found in the papers \cite{AGK20}, \cite{Go20}, and \cite{GH20}.

That a model-theoretic study of a wider class of von Neumann algebras (that is, beyond the finite realm) should be possible is hinted at by the existence of the Ocneanu ultraproduct construction, which allows one to take the ultraproduct of a family of \emph{W$^*$-probability spaces}, that is, $\sigma$-finite von Neumann algebras equipped with a faithful, normal state. (The relevant facts about W$^*$-probability spaces needed in this paper are summarized in Section 2.) Motivated by the Ocneanu ultraproduct, Farah and Hart, in an unpublished work, observed that the category of $\sigma$-finite von Neumann algebras forms a so-called \emph{compact abstract theory} (or CAT), which is a logical framework predating the current incarnation of continuous first-order logic.  The first person to axiomatize (in the sense of the previous paragraph) W$^*$-probability spaces in an appropriate continuous first-order language was Dabrowski \cite{Da19}; in particular, the model-theoretic ultraproduct construction for this class corresponds to the Ocneanu ultraproduct construction.  Dabrowski's axiomatization is quite technical and uses a fair amount of modular theory.  A simpler (but less descriptive) axiomatization was given by  Hart, Sinclair, and the first author in \cite{GHS18}.

Now that the class of W$^*$-probability spaces has been established as an axiomatizable class in an appropriate language, it is only natural to begin a thorough model-theoretic study of this class.  In this paper, we initiate this endeavor.  Our main focus will be on studying the class of \emph{existentially closed W$^*$-probability spaces}.  The notion of an existentially closed structure is the model-theoretic generalization of the notion of an algebraically closed field.  Roughly speaking, given structures $\cal M$ and $\cal N$ (in some language) for which $\cal M$ is a substructure of $\cal N$, we say that $\cal M$ is an existentially closed substructure of $\cal N$ (or that $\cal M$ is existentially closed in $\cal N$) if any existential fact about some elements of $\cal M$ which is true in $\cal N$ is also true in $\cal M$.  Considering that we are in the setting of continuous logic, truth in $\cal M$ is really approximate truth.  Thankfully, this syntactic definition of existentially closed substructure can be given a semantic reformulation that aligns much more with the operator-algebraic perspective:  $\cal M$ is an existentially closed substructure of $\cal N$ if and only if there is an embedding of $\cal N$ into an ultrapower $\cal M^\u$ of $\cal M$ for which the restriction of the embedding to $\cal M$ is the usual diagonal embedding of $\cal M$ into its ultrapower.  If $\cal M$ belongs to some class $\frak C$ of structures, we say that the structure $\cal M$ is existentially closed for $\frak C$ if it is existentially closed in all extensions belonging to $\frak C$.

The Robinsonian school of model theory encourages one to understand the class of models of a theory by understanding its class of existentiallly closed models, that is, the models of the theory which are existentially closed for the class of models of the theory.  The study of existentially closed tracial von Neumann algebras was carried out in the papers \cite{GHS13}, \cite{FGHS16}, \cite{Go18}, and \cite{Go21}.  In Section 3 of this paper, we carry out a systematic study of the class of W$^*$-probability spaces.  In Subsection 3.1, we describe some properties of a W$^*$-probability space that are inherited by an existentially closed substructure.  In particular, we show that if the W$^*$-probability space $(M,\varphi)$ is existentially closed in $(N,\psi)$ and $N$ is a type III$_1$ factor, then so is $M$.  This aids us in our study of the class of existentially closed W$^*$-probability spaces in Subsection 3.2, where we show that any such W$^*$-probability space is necessarily a type III$_1$ factor, which generalizes the result that an existentially closed tracial von Neumann algebra is necessarily a type II$_1$ factor.  Other facts about existentially closed tracial von Neumann algebras are generalized to this setting, such as they tensorially absorb the Araki-Woods factor $R_\infty$ (generalizing the fact that any existentially closed II$_1$ factor tensorially absorbs the hyperfinite II$_1$ factor $R$) and that every automorphism of an existentially closed W$^*$-probability space is approximately inner.  

If one restricts one's attention to the class of tracial von Neumann algebras that satisfy the conclusion of the Connes Embedding Problem, that is, that admit a trace-preserving embedding into the tracial ultrapower of the hyperfinite II$_1$ factor $R$, then one obtains the fact that $R$ itself is an existentially closed element of this class.  Since a tracial von Neumann algebra embeds into the tracial ultrapower of $R$ precisely when it has Kirchberg's QWEP property \cite{Ki93}, it is natural to restrict attention to the class of QWEP W$^*$-probability spaces.  In fact, a result of Ando, Haagerup, and Winslow \cite{AHW13} shows that this class of W$^*$-probability spaces can be characterized by admitting an embedding (with expectation) into the Ocneanu ultrapower of the Araki-Woods factor $R_\infty$, or, as we show below, in model-theoretic terms, is a model of the universal theory of $R_\infty$.  We show that $R_\infty$ is an existentially closed QWEP W$^*$-probability space and prove a few further results about this class of structures.  Subsection 3.4 is concerned with the technique of building W$^*$-probability spaces by games, which is a technique (first introduced in the continuous setting in \cite{Go21}) that is very useful when trying to build e.c.\ objects with extra properties.  The section concludes with Subsection 3.5, which contains some open questions about existentially closed W$^*$-probability spaces.

Section 4 contains two further collections of results about the model theory of W$^*$-probability spaces.  The first collection of results concerns the axiomatizability of various classes of type III factors.  It is shown in \cite{AHW13} that, given any $\lambda\in (0,1]$, the Ocneanu ultraproduct of a family of type III$_\lambda$ factors is again a type III$_\lambda$ factor and a factor is of type III$_\lambda$ if its Ocneanu ultrapower is as well.  Model-theoretically, this implies that the class of W$^*$-probability spaces whose underlying von Neumann algebra is a type III$_\lambda$ factor forms an axiomatizable class.  In Section 4.1, we show that the results of our analysis of existentially closed W$^*$-probability spaces implies that the class of type III$_1$ W$^*$-probability spaces necessarily has a $\forall_2$-axiomatization, that is, has a set of axioms of the form $\sup_x\inf_y\theta(x,y)$, where $x$ and $y$ are finite tuples of variables and $\theta$ is a quantifier-free formula.  We show that for a fixed $\lambda\in (0,1)$, the class of III$_\lambda$ factors cannot be axiomatized using two quantifiers but can be axiomatized using three quantifiers.  In none of these cases do we provide explicit axiomatizations but instead use a ``soft'' criterion for establishing the existence of such axiomatizations given by the Keisler Sandwich theorem.

In Subsection 4.2, we study the notion of elementary equivalence of W$^*$-probability spaces.  Two W$^*$-probability spaces are elementarily equivalent if they cannot be distinguished using a first-order sentence.  Using the Keisler-Shelah theorem, this can be given a semantic reformulation, namely that they have isomorphic ultrapowers.  We first show how the result of Boutonnet, Chifan, and Ioana \cite{BCI15} stating that McDuff's family \cite{McD69} of pairwise non-isomorphic separable II$_1$ factors are in fact pairwise non-elementarily equivalent can be used to show that there exist continuum many non-elementarily equivalent separable type III$_\lambda$ W$^*$-probability spaces for any $\lambda\in (0,1)$.  We are currently unable to extend this result to include $\lambda=1$ but are able to identify at least three non-elementarily equivalent separable type III$_1$ W$^*$-probability spaces.  In order to accomplish this, we show that the class of \emph{non-full} III$_\lambda$ factors (for fixed $\lambda\in (0,1]$) is an axiomatizable class, generalizing the theorem of Farah, Hart, and Sherman \cite{FHS142} that the class of type II$_1$ factors with property Gamma is axiomatizable.  This subsection includes a number of interesting open questions about the study of elementary equivalence of W$^*$-probability spaces.

There are two appendices in this paper.  The first appendix contains results about embedding AFD W$^*$-probability spaces into ultraproducts that are needed in various portions of the paper; most of the results in this appendix are unpublished results of Ando and the second author.  The second appendix concerns Keisler's Sandwich Theorem, which is the main model-theoretic tool needed in our axiomatization results appearing in Subsection 4.1.  Since the continuous logic version of this result has never appeared, we include a complete proof of the result here.  Moreover, we present the result using ultrapowers rather than arbitrary elementary extensions in the interest of the operator-algebraic audience. 

We have made every attempt to keep the model-theoretic prerequisites for this paper to a minimum and try to use ``semantic'' definitions and proofs whenever possible.  That being said, on a few occasions, we need to refer to basic model-theoretic terminology, such as elementary equivalence, elementary embedding, or first-order formula.  A short introduction aimed towards operator algebraists (albeit in the language of tracial von Neumann algebras) can be found in Subsections 2.1 and 2.3 of \cite{AGK20}.

\tableofcontents

\section{Preliminaries}

\subsection{Basic facts about W*-probability spaces}

For every von Neumann algebra $M$, we denote by $\|\cdot\|_\infty$ its uniform norm, by $M_\ast$ its predual, by $\mathcal Z(M)$ its center, by $\mathfrak S_{{\rm fn}}(M)$ the set of faithful normal states on $M$, by $\mathcal U(M)$ its unitary group, by $\Ball(M)$ its unit ball with respect to the uniform norm, by $\Aut(M)$ its automorphism group and by $(M, L^2(M), J, L^2(M)^+)$ its standard form (see \cite{Ha73}). Under the identification $M = (M_\ast)^\ast$, the ultraweak topology on $M$ coincides with the weak-$\ast$ topology on $(M_\ast)^*$. A linear map $\Phi : M \to N$ is said to be \emph{normal} if it is ultraweakly continuous.

A W$^*$-\emph{probability space} is a pair $(M, \varphi)$ that consists of a $\sigma$-finite von Neumann algebra $M$ endowed with a faithful normal state $\varphi \in \frak S_{fn}(M)$.  For every $x \in M$, write $\|x\|_\varphi = \varphi(x^*x)^{1/2}$ (resp.\ $\|x\|_\varphi^\sharp = \varphi(x^*x)^{1/2} + \varphi(xx^*)^{1/2}$). On uniformly bounded sets, the topology induced by the norm $\|\cdot\|_\varphi$ (resp.\ $\|\cdot\|_\varphi^\sharp$) coincides the strong (resp.\ $\ast$-strong) operator topology. We denote by $\sigma^\varphi$ the modular automorphism group associated with the state $\varphi$. By definition, the \emph{centralizer} $M_\varphi$ of the state $\varphi$ is the fixed point algebra of $(M, \sigma^\varphi)$. The \emph{continuous core} of $M$ with respect to $\varphi$ is the crossed product von Neumann algebra $\core_\varphi(M) = M \rtimes_{\sigma^\varphi} \bR$.  The natural inclusion $\pi_\varphi: M \to \core_\varphi(M)$ and the strongly continuous unitary representation $\lambda_\varphi: \bR \to \core_\varphi(M)$ satisfy the covariance relation
$$
  \lambda_\varphi(t) \pi_\varphi(x) \lambda_\varphi(t)^*
  =
  \pi_\varphi(\sigma^\varphi_t(x))
  \quad
  \text{ for all }
  x \in M \text{ and all } t \in \bR.
$$
Set $L_\varphi (\bR) = \lambda_\varphi(\bR)\dpr \subset \core_\varphi(M)$. There is a unique faithful normal conditional expectation $E_{L_\varphi (\bR)}: \core_{\varphi}(M) \to L_\varphi(\bR)$ satisfying $E_{L_\varphi (\bR)}(\pi_\varphi(x) \lambda_\varphi(t)) = \varphi(x) \lambda_\varphi(t)$. The faithful normal semifinite weight defined by $f \mapsto \int_{\bR} \exp(-s)f(s) \, {\rm d}s$ on $L^\infty(\bR)$ gives rise to a faithful normal semifinite weight $\Tr_\varphi$ on $L_\varphi(\bR)$ {\em via} the Fourier transform. The formula $\Tr_\varphi = \Tr_\varphi \circ E_{L_\varphi (\bR)}$ extends it to a faithful normal semifinite trace on $\core_\varphi(M)$. Define the \emph{dual action} $\theta^\varphi : \bR \curvearrowright \core_\varphi(M)$ by the formula 
$$
  \theta_s^\varphi(\pi_\varphi(x) \lambda_\varphi(t))
  =
  \exp(- {\rm i}st) \, \pi_\varphi(x) \lambda_\varphi(t)
  \quad
  \text{ for all }
  x \in M \text{ and all } s,t \in \bR.
$$
Then $\theta_\varphi^\varphi : \bR \curvearrowright \core_\varphi(\bR)$ is a trace-scaling action in the sense that $\Tr_\varphi \circ \theta^\varphi_s = \exp(-s) \Tr_\varphi$ for every $s \in \bR$.

Let $\psi \in \frak S_{fn}(M)$ be any other faithful normal state. By Connes'\ Radon--Nikodym cocycle theorem \cite[Th\'eor\`eme 1.2.1]{Co72} (see also \cite[Theorem VIII.3.3]{Ta03a}), there is a $\ast$-strongly continuous map $u : \bR \to \mathcal U(M) : t \mapsto u_t$ such that 
\begin{enumerate}
    \item $u_{s + t} = u_s \sigma^\varphi_s(u_t)$ for all $s, t \in \bR$,
    \item $\sigma_t^\psi(x) = u_t \sigma_t^\varphi(x) u_t^*$ for all $t \in \bR$ and all $x \in M$.
\end{enumerate}
Item $(1)$ says that $u : \bR \to \mathcal U(M)$ is a $1$-cocycle for $\sigma^\varphi$ while Item $(2)$ says that $\sigma^\varphi$ and $\sigma^\psi$ are cohomologous. Then the $\ast$-isomorphism $\Pi_{\varphi, \psi} : \core_\varphi(M) \to \core_\psi(M) : \pi_\varphi(x) u_t\lambda_\varphi(t) \mapsto \pi_\psi(x)  \lambda_\psi(t)$ satisfies $\Pi_{\varphi, \psi} \circ \pi_\varphi = \pi_\psi$, $\Pi_{\varphi, \psi} \circ \theta^\varphi = \theta^\psi \circ \Pi_{\varphi, \psi}$ and $\Tr_\psi \circ \Pi_{\varphi, \psi} = \Tr_\varphi$. Note however that $\Pi_{\varphi,\psi}$ does not map the subalgebra $L_\varphi(\bR) \subset \core_\varphi(M)$ onto the subalgebra $L_\psi(\bR) \subset \core_\psi(M)$. It follows that the triple $(\core_\varphi(M), \theta^\varphi, \Tr_\varphi)$ does not depend on the choice of the faithful normal state $\varphi \in \frak S_{fn}(M)$ and we simply denote it by $(\core(M), \theta, \Tr)$.

Assume now that $M$ is a factor. The restriction of $\theta$ to the center $\mathcal Z(\core(M))$ is called the \emph{flow of weights}. By factoriality of $M$, the flow of weights $\theta : \bR \curvearrowright \mathcal Z(\core(M))$ is ergodic. 
\begin{itemize}
    \item If $\theta : \bR \curvearrowright \mathcal Z(\core(M))$ corresponds to the translation action $\bR \curvearrowright \bR$, then $M$ is semifinite, that is, $M$ is of type I or II.
    \item If $\theta : \bR \curvearrowright \mathcal Z(\core(M))$ is periodic with period $T > 0$, then letting  $\lambda = \exp(-T)$, we say that $M$ is of type III$_\lambda$.
    \item If $\theta : \bR \curvearrowright \mathcal Z(\core(M))$ has no period, then we say that $M$ is of type III$_0$. 
    \item If $\theta : \bR \curvearrowright \mathcal Z(\core(M))$ is trivial, that is, $\mathcal Z(\core(M)) = \bC 1$, then we say that $M$ is of type III$_1$.
\end{itemize}

Next, we define Connes' $S$-invariant $S(M)$ as the intersection
$$S(M) = \bigcap_{\varphi \in \mathfrak S_{{\rm fn}} (M)} \sigma(\Delta_\varphi)$$
where $\sigma(\Delta_\varphi)$ is the spectrum of the modular operator $\Delta_\varphi$ associated with the faithful normal state $\varphi \in \mathfrak S_{{\rm fn}}(M)$. Then $S(M)\setminus \{0\}$ is a closed multiplicative subgroup of $\bR^*_+$ that completely determines the type of $M$. When $M$ is a type III factor, we have that:
\begin{itemize}
    \item $M$ is of type III$_0$ if and only if $S(M) = \{0, 1\}$;
    \item $M$ is of type III$_\lambda$ if and only if $S(M) =  \{0\} \cup \lambda^\bZ$, for $\lambda\in (0,1)$;
    \item $M$ is of type III$_1$ if and only if $S(M) = [0, +\infty)$.
\end{itemize}

We also define Connes' $T$-invariant $T(M)$ as the set of all $t \in \bR$ for which $\sigma_t^\varphi$ is an inner automorphism. By Connes' Radon--Nikodym cocycle theorem, the above definition does not depend on the choice of the faithful normal state $\varphi \in \frak S_{fn}(M)$. Note that $T(M)$ is a subgroup of $\bR$. In case $M$ is not of type III$_0$, then $T(M)$ is a closed subgroup of $\bR$ that completely determines the type of $M$. Indeed, we have that:
\begin{itemize}
    \item $M$ is semifinite if and only if $T(M) = \bR$;
    \item $M$ is of type III$_\lambda$ if and only if $T(M) = \log(\lambda) \bZ$, for $\lambda\in (0,1)$;
    \item $M$ is of type III$_1$ if and only if $T(M) = \{0\}$.
\end{itemize}   

We refer to \cite{Co72, Ta03a} for further details regarding the structure of type III factors. 

Throughout this paper, for $\lambda\in (0,1)$, $(R_{\lambda}, \varphi_{\lambda})$ denotes the Powers factor of type ${\rm III}_{\lambda}$ endowed with its (unique) $\frac{2 \pi}{|\log(\lambda)|}$-periodic faithful normal state. By definition, we have
$$(R_{\lambda}, \varphi_{\lambda}) \cong ( M_2(\bC), \omega_{\lambda})^{\otimes \bb{N}}$$
where $\omega_{\lambda} : M_2(\bC) \to \bb{C}$ is defined by 
$$\omega_{\lambda}\left( \begin{pmatrix}
x_{11} & x_{12} \\
x_{21} & x_{22}
\end{pmatrix}\right) = \frac{\lambda}{1 + \lambda} x_{11} + \frac{1}{1 + \lambda} x_{22}.$$
By Connes' result \cite{Co75b} (see also \cite[Theorem XVIII.1.1]{Ta03b}), $R_{\lambda_i}$ is the unique AFD factor of type ${\rm III}_{\lambda_i}$.  We also let $R_\infty$ denote the Araki-Woods factor.  Combining Connes'\ result \cite{Co85} and Haagerup's result \cite{Ha85} (see also \cite[Theorem XVIII.4.16]{Ta03b}), $R_{\infty}$ is the unique AFD factor of type ${\rm III_1}$ and moreover we have $R_\infty \cong R_{\lambda_1} \otimes R_{\lambda_2}$ whenever $\log(\lambda_1)/\log(\lambda_2)$ is irrational.

We next clarify what we mean by an inclusion of W$^*$-probability spaces.

\begin{defn}
For W$^*$-probability spaces $(M, \varphi)$ and $(N, \psi)$, we say that $(M, \varphi)$ {\em embeds into} $(N, \psi)$, denoted $(M, \varphi) \hookrightarrow (N, \psi)$, if there exist a unital normal $\ast$-embedding $\iota : M \to N$ such that $\psi \circ \iota = \varphi$ and a faithful normal conditional expectation $E : N \to \iota(M)$ such that $\varphi \circ \iota^{-1} \circ E = \psi$.
\end{defn}

In what follows, we identify $M$ with $\iota(M)$, regard $M \subseteq N$ as a von Neumann subalgebra, and assume that $\iota : M \to N$ is simply given by $\iota(x):=x$. In that case, we say that $(M, \varphi) \subseteq (N, \psi)$ is an {\em inclusion} of W$^*$-probability spaces. By \cite[Theorem IX.4.2]{Ta03a}, the following assertions are equivalent:
\begin{enumerate}
    \item $(M, \varphi) \subseteq (N, \psi)$.
    \item The modular automorphism group $\sigma^\psi$ leaves the subalgebra $M \subseteq N$ globally invariant, $\psi |_M = \varphi$, and $\sigma^{\varphi} = \sigma^\psi |_M$.
\end{enumerate}
In that case, $E : N \to M$ is the unique faithful normal conditional expectation such that $\varphi \circ E = \psi$. Moreover, we have
\begin{align*}
M_\varphi &= \{x \in M \mid \forall t \in {\bb R}, \sigma_t^\varphi(x) = x\} \\
&= \{x \in M \mid \forall t \in {\bb R}, \sigma_t^\psi(x) = x\} \quad (\text{since } \sigma^{\varphi} = \sigma^\psi |_M)\\
&\subseteq N_\psi.
\end{align*}

\subsection{Ocneanu ultraproducts of W*-probability spaces}

Let $I$ be any nonempty directed set and $\u$ any nonprincipal ultrafilter on $I$.  Let $(M_i, \varphi_i)_{i\in I}$ be any family of W$^*$-probability spaces. Following \cite{AH12}, define
\begin{align*}
\ell^\infty(I, M_i) &= \left\{ (x_i)_i \in \prod_{i \in I} M_i \mid \sup_{i \in I} \|x_i\|_\infty < +\infty\right\} \\
\mathfrak I_{\u} &= \left\{ (x_i)_i \in \ell^\infty(I, M_i) \mid \lim_{i \to \u} \|x_i\|_{\varphi_i}^\sharp = 0 \right\} \\
\mathfrak M^{\u} &= \left\{ (x_i)_i \in \ell^\infty(I, M_i) \mid  (x_i)_i \, \mathfrak I_{\u} \subset \mathfrak I_{\u} \quad \text{and} \quad \mathfrak I_{\u} \, (x_i)_i \subset \mathfrak I_{\u}\right\}.
\end{align*}
Observe that $\mathfrak I_{\u} \subseteq \mathfrak M^{\u}$. The {\em multiplier algebra} $\mathfrak M^{\u}$ is a C$^*$-algebra and $\mathfrak I_{\u} \subset \mathfrak M^{\u}$ is a norm closed two-sided ideal. Following \cite[\S 5.1]{Oc85}, we define the {\em ultraproduct von Neumann algebra} by $\prod_\u (M_i,\varphi_i):=(M_i, \varphi_i)^{\u} := \mathfrak M^{\u} / \mathfrak I_{\u}$. We note that the proof given in \cite[5.1]{Oc85} for the case when $I = \bN$ and $\u \in \beta(\bN) \setminus \bN$ applies {\em mutatis mutandis}. We denote the image of $(x_i)_i \in \mathfrak M^{\u}$ in $(M_i,\varphi_i)^\u$ by $(x_i)^{\u}$.  

We now focus on the particular case when $(M_i, \varphi_i) = (M, \varphi)$ for some fixed W$^*$-probability space $(M, \varphi)$. In that case, we write $(M, \varphi)^\u = (M^\u, \varphi^\u)$ for the ultraprower of $(M, \varphi)$.
For every $x \in M$, the constant sequence $(x)_i$ lies in the multiplier algebra $\mathfrak M^{\u}$. We then identify $M$ with $(M + \mathfrak I_{\u})/ \mathfrak I_{\u}$ and regard $M \subset M^{\u}$ as a von Neumann subalgebra. The map $E_{\u} : M^{\u} \to M $ given by $E_\u((x_i)^{\u}) = \sigma \text{-weak} \lim_{i \to {\u}} x_i$ is a faithful normal conditional expectation. Moreover, we have $\varphi \circ E_\u = \varphi^\u$. Thus, $(M, \varphi) \subseteq (M, \varphi)^\u$ is an inclusion of W$^*$-probability spaces. Following \cite[\S 2]{Co74a}, set
$$\mathfrak M_{\u}  = \left\{ (x_i)_i \in \ell^\infty(I, M) \mid \lim_{i \to \u} \|x_i \zeta - \zeta x_i\|  = 0, \forall \zeta \in L^2(M) \right\}.$$
Define the {\em asymptotic centralizer} von Neumann algebra by $M_{\u} = \mathfrak M_{\u} / \mathfrak I_{\u}$, which is a von Neumann subalgebra of $M^{\u}$. By \cite[Proposition 2.8]{Co74a} (see also \cite[Proposition 4.35]{AH12}), we have $M_\u = (M' \cap M^\u)_{\psi^\u}$ for every faithful normal state $\psi \in \frak S_{fn}(M)$.

Now let $(Q, \psi) \subseteq (M, \varphi)$ be any inclusion of W$^*$-probability spaces and denote by $E : M \to Q$ the unique faithful normal conditional expectation such that $\psi \circ E = \varphi$. We have $\ell^\infty(I, Q) \subset \ell^\infty(I, M)$, $\mathfrak I_{\u}(Q) \subseteq \mathfrak I_{\u}(M)$ and $\mathfrak M^{\u}(Q) \subseteq \mathfrak M^{\u}(M)$. We then identify $Q^{\u} = \mathfrak M^{\u}(Q) / \mathfrak I_{\u}(Q)$ with $(\mathfrak M^{\u}(Q) + \mathfrak I_{\u}(M)) / \mathfrak I_{\u}(M)$ and regard $Q^{\u} \subset M^{\u}$ as a von Neumann subalgebra. Observe that the norm $\|\cdot\|_{\psi^{\u}}$ on $Q^{\u}$ is the restriction of the norm $\|\cdot\|_{\varphi^{\u}}$ to $Q^{\u}$. Observe moreover that $(E(x_i))_i \in \mathfrak I_{\u}(Q)$ for all $(x_i)_i \in \mathfrak I_{\u}(M)$ and $(E(x_i))_i \in \mathfrak M^{\u}(Q)$ for all $(x_i)_i \in \mathfrak M^{\u}(M)$. Therefore, the mapping $E^\u : M^{\u}\to Q^{\u}$ given by $E^\u((x_i)^{\u})= (E(x_i))^{\u}$ is a well-defined conditional expectation satisfying $\psi^{\u} \circ E^\u= \varphi^{\u}$. Thus, $(Q, \psi)^\u \subseteq (M, \varphi)^\u$ is an inclusion of W$^*$-probability spaces.

\subsection{Automorphism group, fullness and w-spectral gap}

Let $M$ be any $\sigma$-finite von Neumann algebra. Recall that for every $\varphi \in (M_\ast)^+$, there is a unique vector $\xi_\varphi \in L^2(M)^+$ such that $\varphi(x) = \langle x \xi_\varphi, \xi_\varphi\rangle$ for all $x \in M$. The group $\Aut(M)$ of all automorphisms of $M$ acts on $M_*$ by $\theta(\varphi)=\varphi \circ \theta^{-1}$ for all $\theta \in \Aut(M)$ and $\varphi \in M_*$. Following \cite{Co74a, Ha73}, the $u$-topology on $\Aut(M)$ is the topology of pointwise norm convergence on $M_*$, meaning that a net $(\theta_i)_{i \in I}$ in $\Aut(M)$ converges to the identity $\id_M$ in the $u$-topology if and only if for all $\varphi \in M_*$ we have $\| \theta_i(\varphi) -\varphi \| \to 0$ as $i \to \infty$. This turns $\Aut(M)$ into a complete topological group. When $M_*$ is separable, $\Aut(M)$ is Polish. Since the standard form of $M$ is unique, the group $\Aut(M)$ also acts naturally on $L^2(M)$ and we have $\theta(\xi_\varphi)=\xi_{\theta(\varphi)}$ for every $\varphi \in M_*^+$. It follows that the $u$-topology coincides with the topology of pointwise norm convergence on $L^2(M)$.

We introduce the following terminology:

\begin{defn}
We say that $\theta \in \Aut(M)$ is
\begin{itemize}
    \item {\em approximately inner} if there exists a net $(u_i)_{i \in I}$ in $\mathcal U(M)$ such that $\Ad(u_i) \to \theta$ in the $u$-topology.
    \item {\em weakly inner} in the sense of \cite{Ma18} if the automorphism $\theta \odot \id$ of $M \odot M^{\op}$ extends to an automorphism of the C$^*$-algebra C$^*_{\lambda \cdot \rho}(M)$ generated by the standard representation $\lambda \cdot \rho : M \odot M^{\op} \to  \mathbf B(L^2(M))$.
\end{itemize}
\end{defn}

Any approximately inner automorphism is weakly inner (see \cite[Theorem 4.1]{AHHM18}) but the converse is not true (see \cite{Ma18}).

 We say that a $\sigma$-finite factor $M$ is {\em full} if whenever $(u_i)_{i \in I}$ is a net in $\mathcal U(Q)$ such that $\Ad(u_i) \to \id_M$ in the $u$-topology, there exists a bounded net $(\lambda_i)_{i \in I}$ in $\bb{C}$ such that $u_i - \lambda_i 1 \to 0$ $\ast$-strongly. When $M$ is tracial, $M$ is full if and only if $M$ does not have property Gamma of Murray and von Neumann. It is known that a full factor is never of type $III_0$ (see \cite[Theorem 2.12]{Co74a}). By \cite[Corollary 3.7]{HMV16}, if $M$ is a full factor, then for any nonprincipal ultrafilter $\u$ on any directed set $I$, we have $M' \cap M^\u = \bb{C}1$. (The converse is also true and follows readily from the definitions.)  It follows from the classification of amenable factors \cite{Co75b, Co85, Ha85} that any factor that is amenable and full is necessarily of type I.

 We say that an inclusion of von Neumann algebras $Q \subset M$ is {\em with expectation} if there exists a faithful normal conditional expectation $E : M \to Q$. Moreover, we say that $Q$ has {\em w-spectral gap} in $M$ if, for any nonprincipal ultrafilter $\u$ on any set $I$, we have $Q' \cap M^\u = (Q' \cap M)^\u$. By \cite[Theorem 4.4]{HMV16}, for any $\sigma$-finite full factor $M$ and any $\sigma$-finite von Neumann algebra $N$, $M$ has w-spectral gap in $M \otimes N$. 
 
 \subsection{W*-probability spaces as metric structures}
 
As mentioned in the introduction, Dabrowski \cite{Da19} introduced a first-order language  for axiomatizing W$^*$-probability spaces.  In this language, the sorts are given by operator norm balls centered at the origin of various natural number radii.  The metric on each ball is given by the norm $\|\cdot\|_\varphi^*$, a relative of the norm $\|\cdot\|_\varphi^\#$ used above, which has the advantage that the state is Lipschitz continuous on each sort.  While one has the natural symbols for scalar multiplication, addition, and adjoint, multiplication is not uniformly continuous on each sort and thus Dabrowski uses ``smeared'' multiplication maps defined using modular theory.  Finally, he includes function symbols for the modular automorphism group (for rational times to keep the language countable) as well as some auxiliary symbols needed to make the axiomatization work.  In this language, it is possible to axiomatize a class of structures which, as a category, is equivalent to the category of W$^*$-probability spaces with inclusions as defined above and for which the model-theoretic ultraproduct corresponds to the Ocneanu ultraproduct.

While quite explicit, Dabrowski's language is very technical and cumbersome.  An alternate axiomatization is given by Hart, Sinclair, and the first author in \cite{GHS18}.  There, the sorts are given by vectors of operator norm at most $N$ that are \emph{$K$-bounded} (in a sense akin to that used in bimodule theory).  On these sorts, the metric is induced by the norm $\|\cdot\|_\varphi^\#$ and then all symbols (including multiplication) are naturally uniformly continuous.  A much simpler axiomatization in this language can be given which once again yields an equivalence of categories capturing the Ocneanu ultraproduct except that an inclusion of models of this theory only corresponds to a normal, state-preserving embedding of von Neumann algebras.  In order to recover the correspondence with the above notion of inclusion of W$^*$-probability spaces, one must add the modular automorphism group to the language, which is a harmless move as the modular automorphism group is definable in this language, as shown in \cite{GHS18}.  However, the definability of the automorphism group is given by an abstract argument using the Beth Definability Theorem and thus this axiomatization lacks the concrete flavor of Dabrowski's axiomatization. 

In this paper, the specific first-order framework for studying W$^*$-probability spaces is not important and the reader can feel free to keep either of these two approaches in mind.

\section{Existentially closed W$^*$-probability spaces}

\subsection{Relative existential closedness}

We begin this subsection by officially defining what it means for a W$^*$-probability space to be existentially closed in another:

\begin{defn}
If $(M,\varphi)\subseteq (N,\psi)$, we say that $(M,\varphi)$ is \emph{existentially closed} (e.c.) in $(N,\psi)$ if and only if 
$$(M, \varphi) \subseteq (N, \psi) \subseteq (M, \varphi)^{\u}$$ so that $(M, \varphi) \subseteq (M, \varphi)^{\u}$ is the diagonal inclusion.
\end{defn}

\begin{remarks}

\

\begin{enumerate}
    \item As discussed in the introduction, the previous definition is not the usual ``syntactic'' definition of existential closedness.  Stated in syntactic terms, $(M,\varphi)$ is e.c. in $(N,\psi)$ if and only if:  for every existential formula $\theta$ with parameters from $M$, we have $\theta^{(M,\varphi)}=\theta^{(N,\psi)}$.  The above definition is convenient for operator algebraists who do not wish to understand the precise definition of existential sentence.
    \item The previous definition is rather vague as to the nature of the ultrafilter $\u$.  If $N$ in the previous definition is a separably acting von Neumann algebra, then $\u$ can be taken to be any nonprincipal ultrafilter on $\bb N$.  For general $N$, one needs to take a particular kind of ultrafilter (known as a \emph{good ultrafilter}) on some potentially large index set (depending on the density character of the metric associated to $\|\cdot\|_\psi$).
    \end{enumerate}
\end{remarks}

The following flexibility result allows us to change states when dealing with relatively existentially closed W$^*$-probability spaces.

\begin{prop}\label{prop:flexibility}
Suppose that $(M,\varphi)$ is e.c.\! in $(N,\psi)$.  Then for any faithful normal state $\rho$ on $M$, we have that $(M,\rho)$ is e.c.\! in $(N,\rho\circ E)$, where $E : N \to M$ is the unique faithful normal conditional expectation such that $\varphi\circ E=\psi$. 
\end{prop}

\begin{proof}
By assumption, we have $(M, \varphi) \subseteq (N, \psi) \subseteq (M, \varphi)^{\u}$ so that $(M, \varphi) \subseteq (M, \varphi)^{\u}$ is the diagonal inclusion. Denote by $F : M^{\u} \to N$ the unique conditional expectation such that $\psi \circ F = \varphi^{\u}$. Recall from above that $E_{\u} : M^{\u} \to M$ given by $E_\u((x_i)^{\u})= w\text{-}\lim_{\u} x_i$ is the canonical faithful normal conditional expectation. By definition, we have $\varphi \circ E_{\u} = \varphi^{\u}$. Since $\varphi$ is faithful and since 
$$\varphi \circ (E \circ F) = \psi \circ F = \varphi^{\u} = \varphi \circ E_{\u},$$ by uniqueness of the conditional expectation, we have $E \circ F = E_{\u}$. 

Let now $\rho$ be any faithful normal state on $M$. We then have $(\rho \circ E) \circ F = \rho \circ E_{\u} = \rho^{\u}$. With respect to the same inclusions and the same conditional expectations $E$ and $F$, we have $(M, \rho) \subseteq (N, \rho \circ E)$ and $(N, \rho \circ E) \textcolor{red}{\subseteq} (M^{\u}, (\rho \circ E) \circ F) = (M, \rho)^{\u}$ and so $(M, \rho) \subseteq (N, \rho \circ E) \subseteq (M, \rho)^{\u}$. This shows that $(M,\rho)$ is e.c.\! in $(N,\rho\circ E)$. 
\end{proof}

\begin{lem}\label{basicecfactrs}
Suppose that $(M, \varphi)$ is e.c.\! in $(N, \psi)$.  Then:
\begin{enumerate}
    \item If $N$ is a factor, then $M$ is a factor.  
    \item $T(M)\subseteq T(N)$.
\end{enumerate}
\end{lem}

\begin{proof}
By assumption, we have $(M, \varphi) \subseteq (N, \psi) \subseteq (M, \varphi)^\u$ where $(M, \varphi) \subseteq (M, \varphi)^\u$ is the diagonal inclusion.

(1) Assume that $N$ is factor. Since $\mathcal Z(M) \subseteq \mathcal Z(N)$, it follows that $\mathcal Z(M) = \bC 1$ and so $M$ is a factor.

(2) Let $t \in T(M)$. Then $\sigma_t^\varphi \in \Inn(M)$ and there exists $u \in \mathcal U(M)$ such that $\sigma_t^\varphi = \Ad(u)$. Since $\sigma_t^\psi = \sigma_t^{\varphi^\u}|_N$ and since $\sigma_t^{\varphi^\u} = \Ad(u)$, it follows that $\sigma_t^\psi = \Ad(u) \in \Inn(N)$. Therefore, $t \in T(N)$.
\end{proof}

The next key result gives a necessary condition for an e.c.\! inclusion of W$^*$-probability spaces to be of type III$_0$. Unlike other results in this paper, the next theorem is a purely type III von Neumann algebraic statement.

\begin{thm}\label{ecintypethreezero}
Suppose that $(M, \varphi)$ is e.c.\! in $(N, \psi)$. If $M$ is a type ${\rm III_0}$ factor and $N$ is a factor, then $N$ is a type ${\rm III}_0$ factor.
\end{thm}

\begin{proof}
By assumption, we have $(M, \varphi) \subseteq (N, \psi) \subseteq (M, \varphi)^\u$ where $(M, \varphi) \subseteq (M, \varphi)^\u$ is the diagonal inclusion. Fix faithful normal conditional expectations $E : N \to M$ and $F : M^\u \to N$ so that $E \circ F : M^\u \to M$ is the canonical faithful normal conditional expectation $E_{\u} : M^\u \to M$. 

Since $M$ is a type III$_0$ factor, by \cite[Lemme 5.3.2]{Co72}, there exists a lacunary faithful normal strictly semifinite weight $\Phi$ with infinite multiplicity on $M$. Then the centralizer $M_{\Phi}$ is a type II$_\infty$ von Neumann algebra and there exists a unique faithful normal $\Phi$-preserving conditional expectation $E_{\Phi} : M \to M_{\Phi}$. By \cite[Th\'eor\`eme 5.3.1]{Co72}, there exist $0 < \lambda_0 < 1$ and $u \in M(\sigma^{\Phi}, (-\infty, \log(\lambda_0)])$ such that $u M_{\Phi} u^* = M_{\Phi}$ and $M$ is generated by $M_{\Phi}$ and $u$.  For the definition of the spectral subspace $M(\sigma^{\Phi}, (-\infty, \log(\lambda_0)])$, we refer to \cite[Subsection 2.2]{AH12}. Moreover, we canonically have $M = M_{\Phi} \rtimes_\theta {\bb Z}$ where $\theta \in \Aut(M_{\Phi})$ is given by $\theta = \Ad(u)|_{M_{\Phi}}$. Also, $M_{\Phi}$ has a diffuse center and $\theta |_{\mathcal Z(M_{\Phi})} \in \Aut(\mathcal Z(M_{\Phi}))$ is ergodic. Letting $\tau = \Phi |_{M_{\Phi}}$, we have $\tau(\theta(x)) \leq \lambda_0 \tau(x)$ for every $x \in (M_{\Phi})_+$.

Set $\Psi = \Phi \circ E$. Following \cite[Definition 4.25]{AH12}, we have $\Psi \circ F = \Phi \circ E \circ F = \Phi \circ E_{\u} = \Phi^{\u}$. Since $\Phi$ is lacunary, by \cite[Proposition 4.27]{AH12}, the equality $(M^{\u})_{\Phi^{\u}} = (M_\Phi)^{\u}$ holds. Then \cite[Proposition 6.23]{AH12} shows that $M^\u$ is generated by $(M_\Phi)^\u$ and $u$. Moreover, we canonically have $M^\u = (M_\Phi)^\u \rtimes_{\theta^\u} {\bb Z}$, where $\theta^\u \in \Aut((M_\Phi)^\u)$ is given by $\theta^\u = \Ad(u)|_{(M_\Phi)^\u}$. Letting  $\tau^\u = \Phi^\u|_{(M_\Phi)^\u}$, we have $\tau^\u(\theta^\u(x)) \leq \lambda_0 \tau^\u(x)$ for every $x \in ((M_\Phi)^\u)_+$.

 Note that $N \subseteq M^\u$ is globally invariant under $\sigma^{\Phi^\u}$ and $\sigma^\Psi = \sigma^{\Phi^\u}|_N$. This implies that $M_\Phi \subseteq N_\Psi \subseteq (M_\Phi)^\u$ which further implies that the centralizer $N_\Psi$ is a type ${\rm II_\infty}$ von Neumann algebra with diffuse center. Observe that $N_\Psi = N \cap (M_\Phi)^\u$ and that $\theta^\u(N_\Psi) = u N_\Psi u^* = N_\Psi$. Then $N_\Psi \subseteq (M_\Phi)^\u$ is a ${\bb Z}$-globally invariant von Neumann subalgebra and we have $N_\Psi \rtimes_{\theta^\u} {\bb Z} \subseteq N$.

\begin{claim}\label{claim:crossed-product}
The equality $N = N_\Psi \rtimes_{\theta^\u} {\bb Z}$ holds.
\end{claim}

Recall that $M^\u = (M_\Phi)^\u \rtimes_{\theta^\u} {\bb Z}$ and denote by $E_{\Phi^\u} : M^\u \to (M_\Phi)^\u$ the canonical faithful normal $\Phi^\u$-preserving conditional expectation. We will show below that $E_{\Phi^\u}(N) = N_\Psi$. Once this is proven, \cite[Corollary 3.4]{Su18} implies that 
$$N \subseteq  \left\{x \in (M_\Phi)^\u \rtimes_{\theta^\u} {\bb Z} \mid \forall j \in {\bb Z}, E_{\Phi^\u}(x u^{-j}) \in N_\Psi \right\} = N_\Psi \rtimes_{\theta^\u} {\bb Z}$$
and thus $N = N_\Psi \rtimes_{\theta^\u} {\bb Z}$. 

First observe that $N_\Psi = E_{\Phi^\u}(N_\Psi) \subset E_{\Phi^\u}(N)$. Next, let $x \in  (M_\Phi)^\u \rtimes_{\alg} {\bb Z}$ be any element in the algebraic crossed product. Denote by $\mathcal F \subset {\bb Z}$ the finite support of $x$. Since $\mathcal Z(M_\Phi) = L^\infty(X, \nu)$ is diffuse and since the action ${\bb Z} \curvearrowright \mathcal Z(M_\Phi)$ is ergodic, it follows that the action ${\bb Z} \curvearrowright (X, \nu)$ is essentially free. Then \cite[Lemma 3.1]{Su18} (whose proof works for arbitrary diffuse probability spaces) implies that there exists a finite partition of unity $\sum_{i} p_i = 1$ with projections $p_i \in \mathcal Z(M_\Phi)$ such that $p_i \theta^j(p_i) = 0$ for all $i$ and all $j \in \mathcal F \setminus \{0\}$. Then we have
\begin{align*}
     \sum_i p_i x p_i &= \sum_i \sum_{j \in \mathcal F} p_i \, E_{\Phi^\u}(x u^{-j}) u^j \, p_i \\
     &= \sum_i \sum_{j \in \mathcal F} p_i E_{\Phi^\u}(x u^{-j}) \theta^j(p_i) u^j \\
     &= \sum_i \sum_{j \in \mathcal F} p_i \theta^j(p_i) E_{\Phi^\u}(x u^{-j})  u^j  \quad (\text{since } \theta^j(p_i) \in \mathcal Z(M_\Phi) \subseteq \mathcal Z((M_\Phi)^\u)) \\
     &= \sum_i p_i \, E_{\Phi^\u}(x) \\
     &= E_{\Phi^\u}(x).
\end{align*}
 Thus, we have $E_{\Phi^\u}(x) = \sum_i p_i x p_i$. 
 
We use an idea in \cite[Lemma 3.2]{Su18}. For this, we choose a faithful normal state $\rho$ on $M^\u$ such that $\rho \circ E_{\Phi^\u} = \rho$. Then $(M_\Phi)^\u \subseteq M^\u$ is globally invariant under the modular automorphism group $\sigma^\rho$. Since $\mathcal Z(M_\Phi) \subseteq \mathcal Z((M_\Phi)^\u)$, it follows that $\mathcal Z(M_\Phi)$ is contained in the centralizer $(M^\u)_\rho$. Let now $y \in N$ be any element. For every $n \geq 1$, we may choose $x_n \in (M_\Phi)^\u \rtimes_{\alg} {\bb Z}$ so that  $\|y - x_n\|_{\rho} \leq \frac{1}{2n}$. We then have $\|E_{\Phi^\u}(y - x_n)\|_{\rho} \leq \|y - x_n\|_{\rho} \leq \frac{1}{2n}$. The above reasoning shows that there exists a finite partition of unity $\sum_{i} p^n_i = 1$ with projections $p^n_i \in \mathcal Z(M_\Phi)$ such that $E_{\Phi^\u}(x_n) = \sum_i p^n_i x p^n_i$. Since for every $i$, we have $p^n_i \in \mathcal Z(M_\Phi) \subset (M^\u)_\rho$ and $0 \leq p^n_i \leq 1$, we obtain
 \begin{align*}
     \| \sum_i p^n_i (x_n - y) p^n_i\|_{\rho}^2 &= \rho(\sum_i p^n_i (x_n - y)^* p^n_i (x_n - y) p^n_i) \\
     &\leq \rho(\sum_i p^n_i (x_n - y)^* (x_n - y) p^n_i) \\
     &= \rho(\sum_i p^n_i (x_n - y)^* (x_n - y)) \\
     &= \rho( (x_n - y)^* (x_n - y)) \\
     &= \|x_n - y\|^2_{\rho}.
 \end{align*}
 This implies that
 \begin{align*}
     \|E_{\Phi^\u}(y) - \sum_i p^n_i y p^n_i\|_{\rho} &\leq \|E_{\Phi^\u}(y - x_n)\|_{\rho} 
     + \|E_{\Phi^\u}(x_n) - \sum_i p^n_i x_n p^n_i\|_{\rho} \\
     & \quad + \| \sum_i p^n_i (x_n - y) p^n_i\|_{\rho} \\
     &\leq \|x_n - y\|_{\rho} + 0 + \|x_n - y\|_{\rho} \\
     &\leq \frac{1}{2n} + \frac{1}{2n} = \frac1n.
 \end{align*}
 For every $n \geq 1$, set $y_n := \sum_i p^n_i y p^n_i \in N$ (recall that $p_i^n \in \mathcal Z(M_\Phi) \subseteq N$). Since $\lim_n \|E_{\Phi^\u}(y) - y_n\|_{\rho} = 0$, the uniformly bounded sequence $(y_n)_{n \geq 1}$ converges strongly to $E_{\Phi^\u}(y)$. Since $y_n \in N$ for every $n \geq 1$ and since $N$ is strongly closed, this implies that $E_{\Phi^\u}(y) \in N$. This further implies that $E_{\Phi^\u}(N) \subset N \cap (M_\Phi)^\u = N_\Psi$ and so $E_{\Phi^\u}(N) = N_\Psi$. This finishes the proof of the claim.

By Claim \ref{claim:crossed-product}, we have $N = N_\Psi \rtimes_{\theta^\u} {\bb Z}$ where $N_\Psi$ is a type II$_\infty$ von Neumann algebra with diffuse center. Since $N$ is a factor, we have that $\theta^\u|_{\mathcal Z(N_\Psi)} \in \Aut(\mathcal Z(N_\Psi))$ is necessarily ergodic. Moreover, we have $\tau^\u(\theta^\u(x)) \leq \lambda_0 \tau^\u(x)$ for every $x \in (N_\Psi)_+$. Then \cite[Proposition 5.1.1]{Co72} implies that $N$ is a type III$_0$ factor.
\end{proof}

\begin{cor}\label{ecintypethreeone}
Suppose that $(M, \varphi)$ is e.c.\! in $(N, \psi)$. If $N$ is a type ${\rm III}_1$ factor, then so is $M$.
\end{cor}

\begin{proof}
Lemma \ref{basicecfactrs}(1) implies that $M$ is a factor.  Theorem \ref{ecintypethreezero} implies that $M$ is not of type III$_0$.  Lemma \ref{basicecfactrs}(2) implies that $T(M)\subseteq T(N)=\{0\}$.  Since $M$ is not of type III$_0$, this further implies that $M$ has type III$_1$.
\end{proof}

\subsection{Global existential closedness}

As stated in the introduction, we say that $(M,\varphi)$ is \emph{existentially closed} (e.c.) if whenever $(M,\varphi)\subseteq (N,\psi)$, then $(M,\varphi)$ is e.c. in $(N,\psi)$.
An immediate consequence of Proposition \ref{prop:flexibility} is the following:
\begin{prop}\label{flexible}
If $M$ is a $\sigma$-finite von Neumann algebra, then for any two faithful, normal states $\varphi_1$ and $\varphi_2$ on $M$, we have that $(M,\varphi_1)$ is e.c.\! if and only if $(M,\varphi_2)$ is e.c.
\end{prop} 

By the previous proposition, it is sensible to call a $\sigma$-finite von Neumann algebra $M$ existentially closed (e.c.) if $(M,\varphi)$ is an e.c. W$^*$-probability space for some (equivalently any) $\varphi\in \frak S_{fn}(M)$.  

The next result enumerates many important facts about \textcolor{red}{e.c.}\ W$^*$-probability spaces.  

\begin{thm}\label{thm:ec}
Suppose that $M$ is an e.c.\! W$^*$-probability space.  Then the following assertions hold:
\begin{enumerate}
    \item $M$ is a type III$_1$ factor.
    \item $M\otimes R_\infty\cong M$.
    \item For any full factor $Q$, either $Q$ is of type I or $M \ncong Q \otimes R_\infty$.
    \item Any automorphism of $M$ is approximately inner.
    \item For every subfactor $N \subset M$ with expectation and with w-spectral gap, we have $(N'\cap M)'\cap M=N$.
\end{enumerate}
\end{thm}

\begin{proof}
(1) Choose any faithful normal state $\rho$ on $R_\infty$ and denote by $(N, \psi) := (M, \varphi)*(R_\infty, \rho)$ the corresponding free product von Neumann algebra.  Since $N$ is a type III$_1$ factor (see \cite[Theorem 4.1]{Ue10}), Corollary \ref{ecintypethreeone} implies that $M$ is a type III$_1$ factor.

(2) Fix $\lambda \in (0, 1)$. Then $(M,\varphi)\subseteq (M,\varphi)\otimes (R_{\lambda}, \varphi_{\lambda}) \hookrightarrow (M,\varphi)^\u$ with the composition being the diagonal embedding.  In particular, $(R_{\lambda}, \varphi_{\lambda})$ embeds into $(M'\cap M^\u, \dot\varphi^\u)$ where $\dot\varphi^\u = \varphi^\u|_{M' \cap M^\u}$. This implies that $\lambda \in \sigma_p(\Delta_{\dot\varphi^\u})$. Then \cite[Theorem 4.32]{AH12} shows that $M \otimes R_\lambda \cong M$. Since this is true for every $\lambda \in (0, 1)$, it follows that $M \otimes R_\infty \cong M$.

(3) Assume that there exists a full factor $Q$ such that $M = Q \otimes R_\infty$. Denote by $\alpha \in \Aut( M \otimes M)$ the flip automorphism defined by $\alpha(x \otimes y) = y \otimes x$. Regard $M \subseteq M \otimes M : x \mapsto x \otimes 1$ and set $N := (M \otimes M) \rtimes_\alpha \bb{Z}/2\bb{Z}$. Then $M \subseteq N \hookrightarrow M^\u$ with the composition being the diagonal embedding. Denote by $u  = (u_n)^\u \in \mathcal U(N) \subset \mathcal U(M^\u)$ the canonical unitary implementing the action $\bb{Z}/2 \bb{Z} \curvearrowright^\alpha M$. Then for every $x \in M$, we have $1 \otimes x = \alpha(x \otimes 1) = u(x \otimes 1)u^*$. This implies that $uMu^* \subseteq M' \cap M^\u$. Since $Q$ is a full factor, $Q$ has w-spectral gap in $M = Q \otimes R_\infty$ and so $Q' \cap M^\u = (Q' \cap M)^\u = R_\infty^\u$ (see \cite[Corollary 3.7]{HMV16}). We obtain $uMu^* \subseteq R_\infty^\u$ and so $M \subseteq u^* R_\infty^\u u \subseteq M^\u$. Fix a faithful normal state $\psi$ on $Q$ and consider the faithful normal conditional expectation $E :M \to R_\infty$ defined by $E =  \psi \otimes \id_{R_\infty}$. For every $n$, set $R_n = u_n^* R_\infty u_n \subseteq M$ and define the faithful normal conditional expectation $ E_n : M \to R_n$ by the formula $E_n = \Ad(u_n^*) \circ E \circ \Ad(u_n)$. Up to changing the net, we may assume that for every $x \in M$, we have $E_n(x) - x \to 0$ $\ast$-strongly as $n \to \infty$. For every $n$, we may choose a faithful normal state $\varphi_n$ on $R_n$ so that $(R_n, \varphi_n)$ is an AFD W*-probability space (see Definition \ref{defn:AFD}). Therefore, we may find nets of normal ucp maps $S_j : M \to M_{n_j}(\bb{C})$ and $T_j : M_{n_j}(\bb{C}) \to M$ such that for every $x \in M$, we have $(T_j \circ S_j)(x) - x \to 0$ $\ast$-strongly as $j \to \infty$. Then \cite[Theorem XV.3.1]{Ta03b} implies that $M = Q \otimes R_\infty$ is amenable. Then $Q$ is amenable and full and so $Q$ is a type I factor.

(4) Fix $\theta \in \Aut(M)$ and denote by $N := M \rtimes_\theta \bb{Z}$ the corresponding crossed product von Neumann algebra. Then $M\subseteq N \hookrightarrow M^\u$ with the composition being the diagonal embedding. Denote by $u  = (u_n)^\u \in \mathcal U(N) \subset \mathcal U(M^\u)$ the canonical unitary implementing the action $\bb{Z} \curvearrowright^\theta M$. Then for every $x\in M$, we have $\theta(x) = u x u^*$ and \cite[Theorem 4.1 $(\rm iv) \Rightarrow (\rm v)$]{AHHM18} shows that $\theta$ is  weakly inner. Since $M \cong M \otimes R_\infty$, \cite[Theorem F]{Ma18} implies that $\theta$ is approximately inner.

(5) It suffices to show that $(N' \cap M)' \cap M \subseteq N$. Define the amalgamated free product von Neumann algebra $Q := M \ast_N (L(\bb{Z}) \otimes N)$ with respect to the natural faithful normal conditional expectations $E : M \to N$ and $\tau_{\bb{Z}} \otimes \id_N : L(\bb{Z}) \otimes N \to N$. Then $M \subseteq Q \hookrightarrow M^\u$ with the composition being the diagonal embedding. Denote by $u \in \mathcal U(L(\bb{Z})) \subset \mathcal U(M^\u)$ the canonical Haar unitary. Then we have $u \in N' \cap Q \subseteq N' \cap M^\u = (N' \cap M)^\u$ and so we may write $u = (u_n)^\u$ where $u_n \in \mathcal U(N' \cap M)$ for every $n$. Now take now $b \in (N' \cap M)' \cap M$ and note that $b u_n = u_n b$ for every $n$ and so $b u = u b$. Since $L(\bb{Z})$ is diffuse, we have $L(\bb{Z}) \npreceq_{L(\bb{Z}) \otimes N} N$ in the sense of Popa's intertwining theory (see \cite{Po01, HI15}), whence \cite[Proposition 3.3]{Ue12} implies that $L(\bb{Z})' \cap Q = L(\bb{Z}) \otimes N$. Thus, we obtain $b \in M \cap (L(\bb{Z}) \otimes N) = N$. 
\end{proof}

\begin{remark}
All of the items in the previous theorem are appropriate generalizations of the corresponding facts about e.c.\  tracial von Neumann algebras.  Indeed, suppose that $M$ is an e.c. tracial von Neumann algebra.  Then the finite analog of (1) states that $M$ is a II$_1$ factor, which was proven in \cite{GHS13}, while the finite analog of (2) states that $M$ is McDuff, that is, tensorially absorbs the hyperfinite II$_1$ factor $R$, which was also proven in \cite{GHS13}.  The finite analog of (3) is that $M$ is not a strongly McDuff factor, where a strongly McDuff factor is one that is isomorphic to a factor of the form $Q\otimes R$, where $Q$ is a II$_1$ factor without property Gamma; this fact was proven in \cite[Proposition 6.2.11]{AGK20}.  The finite analogs of (4) and (5) have identical statements and were proven in \cite{FGHS16} and \cite[Proposition 5.16]{Go18} respectively.   
\end{remark}

Proposition \ref{prop:prime} in the appendix states that the ultrapower of a type III$_1$ factor is always a prime factor, that is, cannot be written as the tensor product of diffuse factors.  Combined with Theorem \ref{thm:ec}(2), we immediately obtain:

\begin{cor}
The class of e.c. W$^*$-probability spaces is not closed under ultrapowers.
\end{cor}

In particular, the class of e.c. W$^*$-probability spaces is not axiomatizable.  In model-theoretic language, this means:

\begin{cor}\label{nomodcomp}
The theory of W$^*$-probability spaces does not have a model companion.
\end{cor}

The analogous fact for tracial von Neumann algebras also holds and was the main result of \cite{GHS13}.

\subsection{The case of QWEP factors}

In \cite{FGHS16}, the authors consider the e.c. elements of the class of tracial von Neumann algebras that admit a trace-preserving embedding into the tracial ultrapower $R^\u$ of the hyperfinite II$_1$ factor $R$.  That this is a model-theoretically sensible thing to consider is substantiated by the basic fact that a tracial von Neumann algebra embeds into $R^\u$ if and only if it is a model of the \emph{universal theory} of $R$, denoted $\Th_\forall(R)$, consisting of all conditions of the form $\sup_x\varphi(x)=0$ with $\varphi(x)$ a quantifier-free formula.  As in the unrestricted case, any e.c. model of $\Th_\forall(R)$ is a McDuff II$_1$ factor with only approximately inner automorphisms.  In this case, however, one can name a concrete e.c. object, namely $R$ itself.  In fact, a positive solution to the Connes Embedding Problem is equivalent to the statement that $R$ is an e.c. tracial von Neumann algebra.

In this subsection, we consider the analogous situation for W$^*$-probability spaces.  To motivate the move that is to follow, we recall that a tracial von Neumann algebra embeds into $R^\u$ if and only if it has Kirchberg's QWEP property \cite{Ki93}.  Thus, it appears that the natural course of action to take in our current context is to consider restricting to the class of QWEP W$^*$-probability spaces.  To see that, once again, this is a natural move from the model-theoretic perspective, we recall the main result of \cite{AHW13}, which states that a von Neumann algebra has QWEP if and only if it embeds into $R_\infty^\u$.  (This is technically proven in the case that the von Neumann algebra under consideration is separably acting and the ultrafilter is a nonprincipal ultrafilter on $\bb N$; this result naturally extends to all QWEP von Neumann algebras by writing them as an increasing union of separably acting QWEP subalgebras and using an ultrafilter on a larger index set.) To see that this latter condition has model-theoretic meaning, we make the following observation.  Recall from the introduction that two W$^*$-probability spaces $(M,\varphi)$ and $(N,\psi)$ are \emph{elementarily equivalent} if and only if there are ultrafilters $\u$ and $\cal V$ such that $(M,\varphi)^\u\cong (N,\psi)^{\cal V}$. 

\begin{prop}\label{IIIee}
For any type ${\rm III}_1$ factor $M$ and faithful normal states $\varphi$ and $\psi$ on $M$, we have that $(M,\varphi)$ and $(M,\psi)$ are elementarily equivalent.
\end{prop}

\begin{proof}
The state space of $M^\u$ is \emph{strictly homogeneous} by \cite[Theorem 4.20]{AH12}.  (We note that this result indeed holds for ultrapowers with respect to arbitrary countably incomplete ultrafilters.) Consequently, there is $u \in \mathcal U(M^\u)$ such that $u \varphi^\u u^* = \psi^\u$; the inner automorphism $\Ad(u)$ thus yields that $ (M^\u, \varphi^\u) \cong (M^\u, \psi^\u)$.
\end{proof}

\begin{remark}
The previous proposition is false for type III$_\lambda$ factors, $\lambda\in(0,1)$, even when only considering the universal theory.  Indeed, if $M$ is a type III$_\lambda$ factor, $\lambda\in (0,1)$, and $\varphi,\psi\in \frak S_{fn}(M)$ are such that $\varphi$ is periodic and $\psi$ is not, then $(M,\psi)$ cannot embed into the ultrapower of $(M,\varphi)$.
\end{remark}

A particular consequence of Proposition \ref{IIIee} is that we may unambiguously speak of the universal theory $\Th_\forall(M)$ of any type III$_1$ factor $M$, by which we mean the unique common universal theory of $(M,\varphi)$ for any $\varphi \in \frak S_{fn}(M)$.  From this point of view, the main result of \cite{AHW13} can be reworded by saying that a W$^*$-probability space $(M,\varphi)$ is QWEP if and only if it is a model of $\Th_\forall(R_\infty)$.

\begin{remark}
Most of the results of the previous subsection continue to hold when restricted to the elementary class of QWEP W$^*$-probability spaces.  More specifically, the first four items of Theorem \ref{thm:ec} as well as Corollary \ref{nomodcomp} hold when restricted to the class of QWEP W*-probability spaces.  We do not know if item (5) of Theorem \ref{thm:ec} holds in this restricted case as it is unknown if the amalgamated free product of QWEP von Neumann algebras remains QWEP. In that respect, it follows from \cite[Corollary B]{HI14} that the free product of QWEP von Neumann algebras remains QWEP. For other permanence properties, we refer to \cite[Proposition 4.1]{Oz03}. 
\end{remark}

As mentioned above, $R$ is an e.c. member of the class of QWEP II$_1$ factors.  We now prove the analogous fact in the setting of W$^*$-probability spaces:

\begin{thm}
The Araki--Woods factor $R_\infty$ is an e.c.\! QWEP W$^*$-probability space.
\end{thm}

\begin{proof}
Let $N$ be an e.c.\! QWEP W*-probability space such that $R_\infty\subseteq N$.  Since $N$ is QWEP, there is an embedding $N\hookrightarrow R_\infty^\u$.  By Corollary \ref{cor:uniqueness-1}, up to conjugating by a unitary, we may suppose that the composite embedding is the diagonal embedding.  This shows that $R_\infty$ is e.c.\! in $N$ and hence is an e.c.\! QWEP W$^*$-probability space.  (Note that we have used the fact that being e.c.\! does not depend on the choice of state when we conjugated by a unitary.)
\end{proof}

\subsection{Building W$^*$-probability spaces by games}

We now introduce a method for building W$^*$-probability spaces first introduced in \cite{Go21} (based on the discrete case presented in \cite{Ho85}).  This method goes under many names, such as \textit{Henkin constructions}, \textit{model-theoretic forcing}, or \textit{building models by games}.

We fix a countably infinite set $C$ of distinct symbols that are to represent generators of a separable W$^*$-probability space that two players (traditionally named $\forall$ and $\exists$) are going to build together (albeit adversarially).
The two players take turns playing finite sets of expressions of the form $\left|\theta(c)-r\right|<\epsilon$, where $c$ is a tuple of variables from $C$, $\theta(c)$ is some atomic formula, and each player's move is required to extend (that is, contain) the previous player's move.  The exact form an atomic formula depends on which language we are considering for describing W$^*$-probability spaces, but in either case, they roughly correspond to expressions of the form $\varphi(p(c))$, where $p$ is some expression involving the *-algebra operations as well as modular automorphisms and $\varphi$ is a generic symbol for the state.  (In the case of Dabrowski's language, one is only allowed to use ``smeared'' multiplication in such expressions.)  These plays of the game are called (open) \emph{conditions}.  The game begins with $\forall$'s move.  Moreover, these conditions are required to be \emph{satisfiable}, meaning that there should be some W$^*$-probability space $(M,\varphi)$ and some tuple $a$ from $M$ such that $\left|\theta(a)-r\right|<\epsilon$ for each such expression in the condition.  We play this game for countably many rounds.
At the end of this game, we have enumerated some countable, satisfiable set of expressions. Provided that the players address a certain ``dense'' set of conditions infinitely often, they can ensure that the play is \emph{definitive}, meaning that the final set of expressions yields complete information about all atomic formulae in the variables $C$ (that is, for each atomic formula $\theta(c)$, there should be a unique $r$ such that the play of the game implies that $\theta(c)=r$) and that this data describes a countable, dense $*$-subalgebra of a unique W$^*$-probability space, which is called the \textit{compiled structure}.

\begin{defn}
Given a property $P$ of W$^*$-probability spaces, we say that $P$ is an \textit{enforceable} property is there a strategy for $\exists$ so that, regardless of player $\forall$'s moves, if $\exists$ follows the strategy, then the compiled structure will have property $P$.
\end{defn}

\begin{fact}\label{conjlemma}

\

\begin{enumerate}
    \item (Conjunction lemma \cite[Lemma 2.4]{Ho85})  If $P_n$ is an enforceable property for each $n\in \mathbb N$, then so is the conjunction $\bigwedge_n P_n$.
    \item (\cite[Proposition 2.10]{Ho85} Being e.c. is enforceable.
\end{enumerate}
\end{fact}

Item (2) in the previous fact indicates the significance of this technique of building W$^*$-probability spaces in connection with the study of e.c. W$^*$-probability spaces.

\begin{defn}
A W$^*$-probability space $(M,\varphi)$ is said to be \textbf{enforceable} if the property of being isomorphic to $(M,\varphi)$ is an enforceable property.
\end{defn}

Clearly, if an enforceable W$^*$-probability space exists, then it is unique.

One can relativize the above context by considering only QWEP W$^*$-probability spaces.  One can then speak of enforceable properties of QWEP W$^*$-probability spaces and ask about the existence of the enforceable W$^*$-probability space.

In the analogous game for II$_1$ factors, it was shown in \cite[Theorem 5.2]{Go21} that a positive solution to the Connes Embedding Problem is equivalent to the statement that $R$ is the enforceable tracial von Neumann algebra and that, when restricted to the context of QWEP tracial von Neumann algebras, $R$ is indeed the enforceable object.  It is worth asking if the same is true in the case of QWEP W$^*$-probability algebras.  The answer is actually negative and rests on the following:

\begin{prop}
There is no faithful normal state $\varphi$ on $R_\infty$ such that $(R_\infty,\varphi)$ embeds into all e.c.\! QWEP W$^*$-probability spaces.
\end{prop}

\begin{proof}
By contradiction, assume that there exists a faithful normal state $\varphi$ on $R_\infty$ such that for any other faithful normal state $\psi$ on $R_\infty$ we have $(R_\infty, \varphi) \hookrightarrow (R_\infty, \psi)$.

It is a standard fact that there exists a faithful normal state $\psi$ on $R_\infty$ for which $(R_\infty)_\psi = {\bb C}$. Since $(R_\infty, \varphi) \hookrightarrow (R_\infty, \psi)$, it follows that $(R_\infty)_\varphi = {\bb C}1$. Next, identify $(R_\infty, \psi) = (R_{\lambda_1} \otimes R_{\lambda_2}, \varphi_{\lambda_1} \otimes \varphi_{\lambda_2})$ for appropriately chosen $\lambda_1$ and $\lambda_2$. Then $\psi$ is an almost periodic state on $R_\infty$ in the sense that the corresponding modular operator $\Delta_{\psi}$ is diagonalizable on $L^2(R_\infty)$. Since $(R_\infty, \varphi) \hookrightarrow (R_\infty, \psi)$, it follows that $\varphi$ is an almost periodic state and so $(R_\infty)_\varphi \neq {\bb C}1$ by \cite[Lemma 2.1]{Ue11}. This is a contradiction.
\end{proof}

\begin{cor}
There is no enforceable QWEP W$^*$-probability space.
\end{cor}

\begin{proof}
Suppose, towards a contradiction, that $(M,\varphi)$ is the enforceable QWEP W$^*$-probability space.  Then $(M,\varphi)$ is e.c.\! and embeds into all e.c.\! QWEP W$^*$-probability spaces (see \cite[Section 6]{Go21}).  In particular $(M,\varphi)$ embeds into $(R_\infty,\psi)$ for all faithful normal states $\psi$ on $R_\infty$.  This implies that $M\cong R_\infty$, and we obtain a contradiction with the previous proposition.
\end{proof}

On the other hand, we do have the following proposition:

\begin{prop}\label{hypenf}
The property of being approximately finite dimensional is an enforceable property of QWEP W$^*$-probability spaces. 
\end{prop}

\begin{proof}
By the Conjunction Lemma, it suffices to show that, given any open condition $\Sigma$, any finite number $c_1,\ldots,c_n$ of constants, and any $\epsilon>0$, there are matrix units $(e_{ij})$ for some matrix algebra and some complex coefficients $\alpha^k_{ij}$ for $k=1,\ldots,n$ such that $\Sigma\cup\{\|c_k-\sum_{ij}\alpha^k_{ij}e_{ij}\|_\varphi^\#<\epsilon\}$ is itself a condition.  However, this follows from the fact that $\Sigma$, being satisfiable in some QWEP W$^*$-probability space, must also be satisfiable in $R_\infty$.
\end{proof}

It is quite interesting that player $\exists$ can always enforce the underlying von Neumann algebra of the compiled W$^*$-probability space to be $R_\infty$ although there is no single state on $R_\infty$ that can be enforced.

An important class of e.c. structures is the class of finitely generic structures.  We end this section by briefly discussing this class.  First, since the class of W$^*$-probability spaces has the joint embedding property (meaning that any two W$^*$-probability spaces can be jointly embedded into a third), it follows from \cite[Corollary 2.16]{Go21} that, for each sentence $\sigma$ in the language of W$^*$-probability spaces, there is a unique real number $r$ such that the property $\sigma=r$ is an enforceable property; in this case, we set $\sigma^f$ to denote this unique $r$.

\begin{defn}
A W$^*$-probability space $(M,\varphi)$ is called \textit{finitely generic} if it satisfies the following two properties:
\begin{enumerate}
    \item For every sentence $\sigma$ in the language of W$^*$-probability spaces, we have that $\sigma^{(M,\varphi)}=\sigma^f$.
    \item For any W$^*$-probability space $(N,\psi)$ elementarily equivalent to $(M,\varphi)$ for which $(M,\varphi)\subseteq (N,\psi)$, we have that this inclusion is an elementary embedding.
\end{enumerate}
\end{defn}

The following was shown in \cite[Section 3]{Go21}, generalizing the corresponding classical results:

\begin{fact}

\

\begin{enumerate}
    \item The property of being finitely generic is enforceable.  In particular, finitely generic structures exist.
    \item Finitely generic structures are e.c.
\end{enumerate}
\end{fact}

In \cite[Corollary 6.4]{FGHS16}, it was shown that $\R$ is a finitely generic QWEP tracial von Neumann algebra.  We prove an analogous result here:

\begin{thm}\label{fingen}
There is a faithful normal state $\varphi$ on $R_\infty$ so that $(R_\infty,\varphi)$ is a finitely generic QWEP von Neumann algebra.
\end{thm}

\begin{proof}
Simply apply the Conjunction Lemma to the fact that both being approximately finite dimensional is enforceable (Proposition \ref{hypenf}) and being finitely generic is enforceable.  
\end{proof}

\subsection{Open questions}

We end this section with some open questions:

As shown in the last subsection, the enforceable QWEP factor does not exist.  As mentioned above, a positive solution to the Connes Embedding Problem is equivalent to the statement that $R$ is the enforceable tracial von Neumann algebra.  Due to the recent negative solution of the Connes Embedding Problem \cite{JNVWY20}, we conclude that $R$ is not the enforceable tracial von Neumann algebra.  However, whether or not the enforceable tracial von Neumann algebra exists is an interesting open question.  The analogous question for W$^*$-probability spaces is also open:

\begin{question}
Does the enforceable QWEP W$^*$-probability space exist?
\end{question}

In connection with Theorem \ref{fingen}, we ask if the analog of Proposition \ref{flexible} holds for finitely generic structures:

\begin{question}
If $(M,\varphi)$ is finitely generic for some $\varphi\in \frak S_{fn}(M)$, do we have that $(M,\psi)$ is finitely generic for \emph{every} $\psi\in \frak S_{fn}(M)$?
\end{question}





For the next question, recall that the class of W$^*$-probability spaces whose underlying von Neumann algebra is a III$_\lambda$ factor (for some fixed $\lambda\in (0,1)$) and whose state is a periodic state forms an inductive class.  This result is clear from the fact that this class is closed under ultraproducts and ultraroots (see \cite[Theorem 6.11]{AH12}) and inductive limits.  Explicit axioms for this class were given by Dabrowski in \cite{Da19}.  Let $T_\lambda$ denote some collection of axioms for this class.  Since $T_\lambda$ is an inductive theory, it has e.c. models.  

\begin{question}
What can we say about the class of e.c. models of $T_\lambda$?  If we restrict to QWEP such objects, is $R_\lambda$, with its unique periodic state, an e.c.\! model?  Does being e.c.\! depend on the state?
\end{question}

We should note that the argument appearing in the proof of Theorem \ref{thm:ec}(2) above shows that e.c. models of $T_\lambda$ tensorially absorb $R_\lambda$ (equipped with its unique periodic state) and the same is true for the e.c. elements of the class of QWEP models of $T_\lambda$.  

We also note that by the main result of \cite{AHW13}, given any $\lambda\in (0,1)$, a von Neumann algebra is QWEP if and only if it embeds into $R_\lambda^\u$ with expectation.  If one could improve this result so that any QWEP model $(M,\varphi)$ of $T_\lambda$ embeds into $(R_\lambda,\varphi_\lambda)^\u$, then one could mimic the proof of \ref{hypenf} to show that being hyperfinite is an enforceable property with respect to the class of QWEP models of $T_\lambda$.  Since being e.c. is also an enforceable property, one could conclude that $(R_\lambda,\varphi_\lambda)$ is an e.c. element of the class of QWEP models of $T_\lambda$.  In fact, we have the following:

\begin{prop}
The following statements are equivalent:
\begin{enumerate}
    \item Every QWEP model of $T_\lambda$ embeds into $(R_\lambda,\varphi_\lambda)^\u$.
    \item Hyperfiniteness is an enforceable property of QWEP models of $T_\lambda$.
    \item $(R_\lambda,\varphi_\lambda)$ is the enforceable QWEP model of $T_\lambda$.
    \item $(R_\lambda,\varphi_\lambda)$ is an e.c. element of the class of QWEP models of $T_\lambda$.
    
\end{enumerate}
\end{prop}

Whether or not the above statements indeed hold, by Corollary \ref{cor:uniqueness-lambda}, $(R_\lambda,\varphi_\lambda)$ is an e.c.\  model of $\Th_\forall(R_\lambda,\varphi_\lambda)$.

In connection with the previous question, one might also ask the following:

\begin{question}
Suppose that $(M,\varphi)$ is a type III$_\lambda$ factor equipped with a periodic state.  Is $(M,\varphi)$ an e.c.\! model of $T_\lambda$ if and only if $M_\varphi$ is an e.c.\! II$_1$ factor?
\end{question}


An even more basic question about models of $T_\lambda$ arises:

\begin{question}
If $M$ is a III$_\lambda$ factor and $\varphi$ and $\psi$ are both periodic states on $M$, must $(M,\varphi)$ and $(M,\psi)$ have the same universal theory?
\end{question}

Our final question returns to the general study of e.c.\! W$^*$-probability spaces:

\begin{question}
Are any two e.c.\! (QWEP) W$^*$-probability spaces elementarily equivalent?
\end{question}

The analogous question for II$_1$ factors remains unsettled but is presumed to have a negative answer.

\section{Other model-theoretic results}

\subsection{Axiomatizability results}

In \cite{AH12}, it was shown that, for any $\lambda\in (0,1]$, the Ocneanu ultraproduct of a family of type III$_\lambda$ factors is once again a type III$_\lambda$ factor and a $\sigma$-finite von Neumann algebra is a type III$_\lambda$ factor if the same is true of an Ocneanu ultrapower.  By the ``soft'' test for axiomatizability (see \cite[Proposition 5.14]{BBHU08}), this says that the class of III$_\lambda$ factors (again, for fixed $\lambda\in (0,1]$) is an axiomatizable class.  However, this test does not give us explicit axioms for these axiomatizable classes.  Nevertheless, it is possible to at least describe the quantifier complexity of such axiomatizations.

In order to state such quantifier complexity results and the test we use for obtaining them, we define a $\forall_n$-sentence to be one of the form 
$$\sup_{x_1}\inf_{x_2}\cdots Q_{x_n}\theta(x_1,\ldots,x_n),$$ where each $x_i$ is a finite block of variables, $\theta$ is a quantifier-free formula, and $Q=\sup$ if $n$ is odd and $Q=\inf$ if $n$ is even.  The notion of a $\exists_n$-sentence is defined analogously, beginning with a block of $\inf$ quantifiers instead of $\sup$ quantifiers.  By using dummy variables, we note that any $\forall_n$- or $\exists_n$-sentence is automatically both $\forall_{n+1}$ and $\exists_{n+1}$.  Finally, we say that an axiomatizable class is $\forall_n$-axiomatizable (resp. $\exists_n$-axiomatizable) if there is a set $T$ of axioms for the class consisting solely of conditions of the form $\sigma=0$ with $\sigma$ a nonnegative $\forall_n$-sentence (resp. a $\exists_n$-sentence).

In classical logic, given an axioimatizable class, the \emph{Keisler Sandwich Theorem} \cite{Ke60} can be used to prove the existence of an axiomatization of a particular kind of quantifier complexity.  In our results below, we use (a modified version of) the Keisler Sandwich Theorem, adapted to the setting of continuous logic.  Since neither the continuous version of the Keisler Sandwich Theorem, nor the variant presented here, have appeared in the literature before, we include a proof in Appendix B.

\begin{defn}
Given structures $M$ and $N$ (in the same language) and $n\geq 1$, an \textbf{$(M,N,n)$-ultrapower sandwich} is a chain of embeddings of the form $$M\hookrightarrow N\hookrightarrow M^{\u_1}\hookrightarrow N^{\u_2}\hookrightarrow \cdots \hookrightarrow Q^{\u_{n-1}},$$ where:
\begin{itemize}
    \item $\u_1,\ldots,\u_{n-1}$ are ultrafilters,
    \item $Q=M$ if $n$ is even and $Q=N$ if $n$ is odd, and
\item all compositions $M^{\u_{k-1}}\hookrightarrow M^{\u_{k+1}}$ and $N^{\u_{k-1}}\hookrightarrow N^{\u_{k+1}}$ are the diagonal embeddings (which in particular implies that $\u_{k+1}\cong \u_{k-1}\otimes \mathcal{V}$ for some ultrafilter $\mathcal{V}$).
\end{itemize}
\end{defn}


\begin{thm}[Keisler Sandwich Theorem (Ultrapower version)]\label{keisler}
Suppose that $T$ is a theory.  Then:
\begin{enumerate}
    \item $T$ is $\forall_n$-axiomatizable if and only if:  whenever there is an $(M,N,n)$-ultrapower sandwich and $N\models T$, then $M\models T$.
    \item $T$ is $\exists_n$-axiomatizable if and only if:  whenever there is an $(M,N,n)$-ultrapower sandwich and $M\models T$, then $N\models T$.
\end{enumerate}
\end{thm}

Note that an $(M,N,2)$-ultrapower sandwich is simply a chain of embeddings $M\hookrightarrow N\hookrightarrow M^\u$ such that the embedding $M\hookrightarrow M^\u$ is the diagonal embedding.  Thus, an $(M,N,2)$-ultrapower sandwich exists precisely when $M$ is e.c. in $N$.  By Theorem \ref{keisler} and Corollary \ref{ecintypethreeone}, we immediately have:

\begin{thm}\label{thm:III1-forall2}
The class of type ${\rm III}_1$ factors is $\forall_2$-axiomatizable.
\end{thm}

\begin{proof}[Alternate proof of Theorem \ref{thm:III1-forall2}]
There is an alternative to showing that a theory $T$ is $\forall_2$-axiomatizable and that is to show that the collection of models of the theory is closed under unions of chains.  This can be established quite easily for the case of III$_1$ factors as follows.  Let $(M_i, \varphi_i)_{i \in I}$ be any increasing chain of type III$_1$ W$^*$-probability spaces. Denote by $(M, \varphi)$ the inductive limit of $(M_i, \varphi_i)_{i \in I}$ in the language of W$^*$-probability spaces. Then for every $i \in I$, we have $(M_i, \varphi_i) \subseteq (M, \varphi)$ and we denote by $E_i : M \to M_i$ the unique faithful normal conditional expectation such that $\varphi_i \circ E_i = \varphi$. Also, we have $\bigvee_{i \in I} M_i = M$.

For every $i \in I$, we may consider the trace-preserving inclusion of continuous cores $ \core_{\varphi_i}(M_i) \subseteq \core_\varphi(M)$ with trace-preserving faithful normal conditional expectation $F_i : \core_\varphi(M) \to \core_{\varphi_i}(M_i)$. Since $M_i$ is a type III$_1$ factor, $\core_{\varphi_i}(M_i)$ is a type II$_\infty$ factor. In order to show that $M$ is a type III$_1$ factor, it suffices to show that $\core_\varphi(M)$ is a factor. Let $x \in \mathcal Z(\core_\varphi(M))$ be any central element. Then for every $i \in I$, we have $F_i(x) \in \mathcal Z(\core_{\varphi_i}(M_i))$ and so $F_i(x) \in \bC 1$. Since $\bigvee_{i \in I} \core_{\varphi_i}(M_i) = \core_\varphi(M)$, it follows that $F_i(x) \to x$ $\ast$-strongly as $i \to \infty$, which implies that $x \in \bC 1$.
\end{proof}

The case of type III$_\lambda$ factors for $\lambda\in (0,1)$ is inherently more complicated:

\begin{prop}\label{prop:not-axiomatizable}
For any $\lambda\in (0,1)$, the axiomatizable class of type ${\rm III}_\lambda$ factors is not $\forall_2$-axiomatizable.
\end{prop}

\begin{proof}
Fix $\lambda\in (0,1)$.  We construct an increasing sequence of W$^*$-probability space type III$_\lambda$ factors $(M_n, \varphi_n)$ so that, setting $(M,\varphi):=\bigvee_{n\in \bb N}(M_n,\varphi_n)$, we have that $M$ is a type III$_1$ factor.

Take $\lambda_1, \lambda_2 \in (0, 1)$ with $\log (\lambda_1)/\log(\lambda_2)$ irrational, whence $R_{\lambda_1}\otimes R_{\lambda_2}\cong R_\infty$.  For every $n \in \bb{N}$, set $(M_n, \varphi_n) = (R_{\lambda_1}, \varphi_{\lambda_1}) \otimes ( M_2(\bC), \omega_{\lambda_2})^{\otimes \{0, \dots, n\}}$. Then $(M_n, \varphi_n)_{n \in \bb{N}}$ is an increasing sequence of W$^*$-probability spaces type ${\rm III}_{\lambda_1}$ factors. Indeed, observe that for every $n \in \bb{N}$, we have $M_n \cong R_{\lambda_1}$. Moreover, for every $n \in \bb{N}$, we have $\varphi_{n + 1}|_{M_n} = \varphi_n$ and the linear mapping $E_n : M_{n + 1} \to M_n$ defined by $E_n = \id_{M_n} \otimes \omega_{\lambda_2}$ is a faithful normal conditional expectation such that $\varphi_n \circ E_n = \varphi_{n + 1}$.

However, we have $ \bigvee_{n \in \bb{N}} (M_n, \varphi_n) \cong (R_{\lambda_1} \otimes R_{\lambda_2}, \varphi_{\lambda_1} \otimes \varphi_{\lambda_2}) =: (M, \varphi) \cong (R_\infty, \varphi)$.  Indeed, for every $n \in \bb{N}$, we have $\varphi|_{M_n} = \varphi_n$ and the linear mapping $F_n : M \to M_n$ defined by $F_n = \id_{M_n} \otimes \omega_{\lambda_2}^{\otimes {\bb N} \setminus  \{0, \dots, n\} }$ is a faithful normal conditional expectation such that $\varphi_n \circ F_n = \varphi$. Thus, the inductive limit of W$^*$-probability spaces $(M, \varphi) = \bigvee_{n \in \bb{N}} (M_n, \varphi_n)$ is a type ${\rm III_1}$ factor.
\end{proof}

We also note:

\begin{prop}
For any $\lambda\in (0,1]$, the class of type ${\rm III}_\lambda$ factors is not $\exists_2$-axiomatizable.
\end{prop}

\begin{proof}
Fix $\lambda\in (0,1]$ and suppose that $M$ is a type III$_\lambda$ factor that is not full (e.g.\ $M=R_\lambda$).  Then there is a $\sigma$-finite von Neumann algebra $N$ such that $N$ is not a factor and yet $M\subseteq N\subseteq M^\u$.  By Theorem \ref{keisler}, the result follows.
\end{proof}

While the class of type III$_\lambda$ factors, $\lambda\in (0,1)$, cannot be axiomatized using two quantifiers, we now show that it can be axiomatized using three quantifiers.  The key to proving this is the following:

\begin{prop}\label{lambdaequalsmu}
Suppose that we have an $(M,N,3)$-sandwich with $M$ a type ${\rm III}_\mu$ factor and $N$ a type ${\rm III}_\lambda$ factor, $\lambda,\mu\in (0,1)$.  Then $\lambda=\mu$.
\end{prop}

\begin{proof}
Set $T_\lambda=\frac{2\pi}{|\log(\lambda)|}$ and $T_\mu=\frac{2\pi}{|\log(\mu)|}$.  Choose a $T_\lambda$-periodic faithful normal state $\psi$ on $N$ and apply Proposition \ref{prop:flexibility} to the inclusion $N \subseteq  M^{\u_1} \subseteq N^{\u_2}$. Since $\psi^{\u_2}|_{M^{\u_1}}$ is $T_\lambda$-periodic, we have $\{0\} \cup \mu^{\bb Z} = S(M^{\u_1}) \subseteq \sigma(\Delta_{\psi^{\u_2}|_{M^{\u_1}}}) \subseteq \{0\} \cup \lambda^{\bb Z}$. Next, choose a $T_\mu$-periodic faithful normal  state $\varphi$ on $M$ and apply Proposition \ref{prop:flexibility} to the inclusion $M \subseteq  N \subseteq M^{\u_1}$. Since $\varphi^{\u_1}|_{N}$ is $T_\mu$-periodic, we have $\{0\} \cup \lambda^{\bb Z} = S(N) \subseteq \sigma(\Delta_{\varphi^{\u_1}|_{N}}) \subseteq \{0\} \cup \mu^{\bb Z}$. This shows that $\mu = \lambda$.
\end{proof}

\begin{cor}
The class of type ${\rm III}_\lambda$ factors is both $\forall_3$- and $\exists_3$-axiomatizable.
\end{cor}

\begin{proof}
Consider a $(M,N,3)$-sandwich.  First suppose that $N$ is a III$_\lambda$ factor.  By Lemma \ref{basicecfactrs}, $M$ is a factor.  Since $M^{\u_1}$ contains $N$ with expectation, $M^{\u_1}$ is type III, whence so is $M$.  By Theorem \ref{ecintypethreezero}, $M$ is not of type III$_0$.  If $M$ is type III$_1$, then so is $M^{\u_1}$; since $N$ is e.c.\ in $M^{\u_1}$, Corollary \ref{ecintypethreeone} implies that $N$ is type III$_1$, a contradiction.  Thus, $M$ is type III$_\mu$ for some $\mu\in (0,1)$, whence $\lambda=\mu$ by Proposition \ref{lambdaequalsmu}.  This shows that the class of III$_\lambda$ factors is $\forall_3$-axiomatizable.

Now suppose that $M$ is a III$_\lambda$ factor.  Since $N$ is e.c.\ in the III$_\lambda$ factor $M^{\u_1}$, we see that $N$ is a factor.  Since $N$ contains the type III factor $M$ with expectation, we see that $N$ has type III.  By Corollary \ref{ecintypethreeone}, $N$ does not have type III$_1$.  Since $N$ is e.c.\ in the III$_\lambda$ factor $M^{\u_1}$, $N$ does not have type III$_0$ by Theorem \ref{ecintypethreezero}.  Thus, $N$ has type III$_\mu$ for some $\mu\in (0,1)$, and thus $\lambda=\mu$ by Proposition \ref{lambdaequalsmu}.  This shows that the class of III$_\lambda$ factors is $\exists_3$-axiomatizable.
\end{proof}



\subsection{First-order theories of W*-probability spaces}

In this final subsection, we consider the task of counting the number of first-order theories of III$_\lambda$ W$^*$-probability spaces for $\lambda\in (0,1]$.

In \cite{BCI15}, Boutonnet, Chifan, and Ioana showed that there exist  continuum many pairwise non-elementarily equivalent separable II$_1$ factors. More precisely, they showed that the family $(M_\alpha)_{\alpha \in 2^{\bN}}$ of separable II$_1$ factors constructed by McDuff in \cite{McD69} provides such a continuum. We observe that their result can be easily applied to construct such a continuum in the realm of type III$_\lambda$ factors for $\lambda\in (0,1)$.

\begin{thm}
Fix $\lambda\in (0,1)$. Then there exist  continuum many non-pairwise elementary equivalent separable type ${\rm III}_\lambda$ factors.
\end{thm}

\begin{proof}
Consider the family $(M_\alpha)_{\alpha \in 2^{\bN}}$ of separable II$_1$ factors constructed by McDuff in \cite{McD69}. For every $\alpha \in 2^{\bN}$, $M_\alpha$ is a McDuff factor, that is, $M_\alpha \otimes R \cong M_\alpha$. This implies that $\lambda \in \mathcal F(M_\alpha)$. Set $M_\alpha^\infty = M_\alpha \otimes B(\ell^2)$ and choose an automorphism $\theta_\alpha^\lambda \in \Aut(M_\alpha^\infty )$ such that $\tau \circ \theta_\alpha^\lambda = \lambda \tau$, where $\tau$ is any faithful normal semifinite trace on $M_\alpha^\infty$. Then $N_\alpha = M_\alpha^\infty \rtimes_{\theta_\alpha^\lambda} \bZ$ is a type III$_\lambda$ factor whose discrete core is isomorphic to $M_\alpha^\infty$ (see \cite[Th\'eor\`eme 4.4.1]{Co72}). 

Now suppose that $\alpha, \beta \in 2^{\bN}$ are such that $N_\alpha$ is elementarily equivalent to $N_\beta$.  By the Keisler-Shelah Theorem, there are ultrafilters $\mathcal U$ and $\mathcal V$ such that $(N_\alpha)^\u \cong (N_\beta)^{\mathcal V}$. By \cite[Proposition 4.7]{AH12}, the discrete core of $(N_\alpha)^\u$ (resp.\! $(N_\beta)^{\mathcal V}$) is $(M_\alpha^\infty)^\u$ (resp.\! $(M_\beta^\infty)^\v$). Then \cite[Th\'eor\`eme 4.4.1]{Co72} implies that $M_\alpha^\u\otimes B(\ell^2) = (M_\alpha^\infty)^\u \cong (M_\beta^\infty)^\v = M_\beta^\v\otimes B(\ell^2)$. Since $M_\alpha^\u$ and $M_\beta^\v$ have full fundamental group, we have $M_\alpha^\u \cong M_\beta^\v$ and \cite{BCI15} further implies that $\alpha = \beta$. 
\end{proof}

The preceding question naturally raises the following:

\begin{question}
Do there exist continuum many non-pairwise elementary equivalent separable type III$_1$ factors?
\end{question}

While we cannot yet answer the above question, we can at least find three such factors.  First, due to the recent negative solution of the Connes Embedding Problem \cite{JNVWY20}, there must exist a non-QWEP type III$_1$ factor $M$. Consequently, $M$ is not a model of $\Th_\forall(R_\infty)$, and thus is not elementarily equivalent to $R_\infty$.  To find a third theory of type III$_1$ factors, we recall that in \cite[3.2.2]{FHS142}, property Gamma was shown to be an axiomatizable property and thus could be used to distinguish theories of II$_1$ factors.  The correct generalization of property Gamma to our context is that of being non-full.  Here, we show that the non-full type III$_\lambda$ factors form an axiomatizable class and, as a consequence, that the non-full factors form a \emph{local class}, that is, closed under ultrapowers and ultraroots (which is still sufficient for differentiating between theories):

\begin{thm}\label{prop:non-full}
For $\lambda \in (0, 1]$, the class of non-full type ${\rm III}_\lambda$ factors is axiomatizable.
\end{thm}

\begin{proof}
We use the aforementioned ``soft'' test for being axiomatizable, that is, we show that the class of non-full type III$_\lambda$ factors is closed under ultraproducts and ultraroots.  We first show the latter.
Let $\u$ be any nonprincipal ultrafilter and let $M$ be any full type III$_\lambda$ factor; we show that $M^\u$ is also full.  To see this, let $\v$ be any nonprincipal ultrafilter. Then we have $(M^\u)^\v = M^{\v \otimes \u}$ by \cite[Proposition 2.4]{AHHM18}.  Since $M$ is a full factor, we have $M' \cap M^{\v \otimes \u} = \bb{C}1$. Since $M \subset M^\u$, we have $(M^\u)' \cap (M^\u)^\v \subseteq M' \cap (M^\u)^\v = M' \cap M^{\v \otimes \u} = \bb{C}1$. Since this holds true for any nonprincipal ultrafilter $\v$, this further implies that $M^\u$ is a full factor.

Conversely, let $(M_i, \varphi_i)_{i \in I}$ be any family of non-full type III$_\lambda$ factor W$^*$-probability spaces and $\u$ a nonprincipal ultrafilter on $I$. Then $(M^\u, \varphi^\u) = (M_i, \varphi_i)^\u$ is a type ${\rm III}_\lambda$ factor.  We show that $M^\u$ is also non-full. It is well-known that we may assume that $\u$ is \emph{countably incomplete}, meaning that there is some countable collection of sets from $\u$ that has empty intersection.  As a result, this allows us to define a sequence $(\epsilon_i)_{i\in I}$ of positive real numbers such that $\lim_\u \epsilon_i=0$.  For every $i \in I$, there is an index set $J_i$ for which there exists a sequence $(u^i_j)_{j \in J_i}$ in $\mathcal U(M_i)$ such that $\lim_j \|u^i_j \varphi_i - \varphi_i u^i_j\| = 0$, $\lim_j \|y u^i_j - u^i_j y\|^\sharp_{\varphi_i} = 0$ for every $y \in M_i$ and $\lim_j \varphi_i(u^i_j) = 0$. Let $\mathcal F = \{X_1, \dots, X_m\} \subset M^\u$ be any finite subset. Write $X_k = (x_{k, i})^\u$ for every $1 \leq k \leq m$. For every $i\in I$, there exists $v_i = u_{j_i} \in \mathcal U(M_i)$ such that $\|v_i \varphi_i - \varphi_i v_i\| \leq \epsilon_i$, $\|v_i x_{k, i} - x_{k, i} v_i\|_{\varphi_i}^\sharp \leq \epsilon_i$ and $|\varphi_i(v_i)| \leq \epsilon_i$ for every $1 \leq k \leq m$. Then we have $(v_i)_{i \in I} \in \mathfrak M^\u(M_i, \varphi_i)$, $v_{\mathcal F} = (v_i)^\u \in \mathcal U((M^\u)_{\varphi^\u})$ and $\varphi^\u(v_{\mathcal F}) = 0$ (see \cite[Proposition 2.4]{Ho14} and \cite[Lemma 4.36]{AH12}). By construction, we have $X v_{\mathcal F} = v_{\mathcal F} X$ for every $X \in \mathcal F$ and $\varphi^\u(v_{\mathcal F}) = 0$. Since this holds true for any finite subset $\mathcal F \subset M^\u$, this further implies that $M^\u$ is not full.
\end{proof}

\begin{cor}
The following three III$_1$ factors are pairwise non-elementarily equivalent:
\begin{enumerate}
    \item $R_\infty$.
    \item Any non-QWEP III$_1$ factor.
    \item Any full QWEP III$_1$ factor.
\end{enumerate}
\end{cor}


\begin{remark}
In \cite{FHS142}, the authors provide explicit axioms for having property Gamma.  It would be interesting to find explicit axioms for the class of non-full III$_\lambda$ W$^*$-probability spaces.
\end{remark}

In \cite[Theorem 4.3]{FHS142}, it is shown that, for any separable II$_1$ factor $M$, there are continuum many pairwise nonisomorphic separable II$_1$ factors elementarily equivalent to $M$.  We ask if the analogous result holds true for W*-probability spaces:

\begin{question}
For any separable W*-probability space $(M,\varphi)$, are there continuum many nonisomorphic separable $(N,\psi)$ elementarily equivalent to $(M,\varphi)$?
\end{question}

By Proposition \ref{IIIee}, if $M$ is any type III$_1$ factor, then $(M,\varphi)$ is elementarily equivalent to $(M,\psi)$ for all $\varphi,\psi\in \frak S_{fn}(M)$.  In this sense, the previous question has a somewhat trivial positive solution for type III$_1$ factors.  It would be more interesting to find continuum many nonisomophic separable models of any given theory of W*-probability spaces whose underlying von Neumann algebras themselves are pairwise nonisomorphic.

\appendix

\section{Embeddings of AFD W*-probability spaces into ultraproducts}\label{sect:appendix}

Most of the results presented in this appendix are due to Ando and Houdayer (unpublished work).

For the ease of exposition, all ultrafilters in this appendix are assumed to be nonprincipal ultrafilters on $\bb N$.  However, all of the results hold verbatim for countably incomplete ultrafilters on arbitrary index sets with only routine modifications of the proofs needed. These more general versions of the results are what are used throughout the main part of the paper.

Let $(M, \varphi)$ be any type III$_\lambda$ factor, where $\lambda \in (0, 1]$, endowed with a faithful normal state such the centralizer $(M^\u)_{\varphi^{\mathcal U}}$ of the ultraproduct state $\varphi^{\mathcal U}$ is a type II$_1$ factor. 
\begin{itemize}
\item If $M$ is of type III$_\lambda$ with $\lambda \in (0, 1)$, we may take any $\frac{2 \pi}{|\log(\lambda)|}$-periodic faithful normal state $\varphi$  on $M$. 
\item If $M$ is of type III$_1$, we may take any faithful normal state $\varphi$ on $M$ (see \cite[Proposition 4.24]{AH12}).
\end{itemize}

First, we observe that any atomic (discrete) W$^*$-probability space has a unique embedding into $(M, \varphi)^\u$ up to unitary conjugacy.

\begin{lem}\label{lem:atomic}
Let $(P, \psi)$ be any atomic W$^*$-probability space. For every $i \in \{1, 2\}$, let $\pi_i : (P, \psi) \hookrightarrow (M, \varphi)^\u$ be any embedding. Then there exists $u \in \mathcal U((M^\u)_{\varphi^{\mathcal U}})$ such that $\pi_2 = \Ad(u) \circ \pi_1$.
\end{lem}

\begin{proof}
Write $P = \bigoplus_p P_p$ where each $P_p$ is a type I factor. For every $p$, 
denote by $(e^p_{kl})_{k,l}$ a system of matrix units for $P_p$ and denote by $(\lambda^p_k)_k$ positive reals such that $\sum_k \lambda_k^p = 1$ and $\frac{\psi(\, \cdot \, 1_{P_p})}{\psi(1_{P_p})} = \tau_{P_p}(\diag(\lambda^p_k) \, \cdot \,)$. For all $k, l, p$ and all $t \in \mathbb R$, we have $\sigma_t^\psi(e^p_{kl}) = (\lambda^p_k / \lambda^p_l)^{{\rm i} t} \, e^p_{kl}$. 

Observe for every $i \in \{1, 2\}$, since $\varphi^\u \circ \pi_i = \psi$ and since $\pi_i(P) \subset M^\u$ is globally invariant under $\sigma^{\varphi^\u}$, we have $\sigma_t^{\varphi^\u} \circ \pi_i = \pi_i \circ \sigma_t^\psi$ for all $t \in \mathbb R$ by \cite[Lemme 1.2.10]{Co72}. Since $(M^\u)_{\varphi^{\mathcal U}}$ is a type II$_1$ factor and since $\pi_1(e^p_{11}), \pi_2(e^p_{11}) \in (M^\u)_{\varphi^{\mathcal U}}$ and $\varphi^\u(\pi_1(e^p_{11})) = \psi(e^p_{11}) = \varphi^\u(\pi_2(e^p_{11}))$, there exists a partial isometry $v_p \in (M^\u)_{\varphi^{\mathcal U}}$ such that $v_p^*v_p = \pi_1(e^p_{11})$ and $v_pv_p^* = \pi_2(e^p_{11})$. If we let $u = \sum_p \sum_k \pi_2(e^p_{k1}) v_p \pi_1(e^p_{1k})$, we have $u \in \mathcal U((M^\u)_{\varphi^{\mathcal U}})$ and $\pi_2(e^p_{kl}) = u \pi_1(e^p_{kl}) u^*$ for all $k, l, p$. Therefore, we have $\pi_2 = \Ad(u) \circ \pi_1$.
\end{proof}

Next, we prove that the unique embedding property into $(M, \varphi)^\u$ up to unitary conjugacy, is stable under taking increasing unions.

\begin{lem}\label{lem:increasing}
Let $(P, \psi)$ be any separable W$^*$-probability space and $(P_n, \psi_n) \subseteq (P, \psi)$ any increasing sequence of W$^*$-probability subspaces such that $\bigvee_{n \in \mathbb N} (P_n, \psi_n) = (P, \psi)$. For every $i \in \{1, 2\}$, let $\pi_i : (P, \psi) \hookrightarrow (M, \varphi)^\u$ be any embedding. Assume that for every $n \in \mathbb N$, there exists $u_n \in \mathcal U((M^\u)_{\varphi^{\mathcal U}})$ such that $\pi_2(x) = u_n \pi_1(x) u_n^*$ for every $x \in P_n$. 

Then there exists $u \in \mathcal U((M^\u)_{\varphi^{\mathcal U}})$ such that $\pi_2 = \Ad(u) \circ \pi_1$.
\end{lem}

\begin{proof}
For every $n \in \mathbb N$, write $u_n := (u_m^n)^\u$ where $(u_m^n)_m \in \mathfrak M^\u(M)$ and $u_m^n \in \mathcal U(M)$ (see e.g.\ \cite[Lemma 2.1]{HI15}). For every $n \in \mathbb N$, denote by $X_n = \{y^n_k \mid k \in \mathbb N\} \subset P_n$ a $\|\cdot\|_{\psi_n}^\sharp$-dense countable subset and set $X_n^{\leq n} := \{y^n_k \mid 0 \leq k \leq n\}$. For every $i \in \{1, 2\}$ and every $b \in X_n$, write  $\pi_i(b) = (b_m^i)^\u$ where $(b^i_m)_m \in \mathfrak M^\u(M)$.

For every $n \in \mathbb N$, define 
\begin{align*}
F_n &:= \bigcap_{0 \leq k \leq n, b \in X_k^{\leq k}} \left\{m \in \mathbb N \mid \|u_m^n \, b_m^1 \, (u_m^n)^*- b_m^2 \|^\sharp_\varphi < \frac{1}{n + 1} \right\} \\
G_n &:= \left\{ m \in \mathbb N \mid m \geq n\right\} \cap \left\{m \in \mathbb N \mid \|u_m^n \varphi - \varphi u_m^n\| <  \frac{1}{n + 1} \right\} \cap \bigcap_{j = 1}^n F_j.
\end{align*}
By construction and since $\u$ is a nonprincipal ultrafilter on $\mathbb N$, $(G_n)_{n \in \mathbb N}$ is a decreasing sequence of subsets of $\u$ such that $G_0 = \mathbb N$ and $\bigcap_{n \in \mathbb N} G_n = \emptyset$. For every $m \in \mathbb N$, set $v_m = u_m^n \in \mathcal U(M)$ where $n \in \mathbb N$ is the unique integer such that $m \in G_n \setminus G_{n + 1}$.

Let $n \in \mathbb N$. If $m \in G_n = \bigcup_{j \geq n} G_j \setminus G_{j + 1}$, denote by $p \geq n$ the unique integer such that $m \in G_p \setminus G_{p + 1}$. Since $v_m = u_m^p$, we have 
\begin{itemize}
\item $\|v_m \varphi - \varphi v_m\| = \|u_m^p \varphi - \varphi u_m^p\| < \frac{1}{p + 1} \leq \frac{1}{n + 1}$ and
\item $\|v_m \, b_m^1 \, v_m^*- b_m^2 \|^\sharp_\varphi = \|u_m^p \, b_m^1 \, (u_m^p)^*- b_m^2 \|^\sharp_\varphi < \frac{1}{p + 1} \leq \frac{1}{n + 1}$ for all $0 \leq k \leq n$ and all $b \in X_k^{\leq k}$.
\end{itemize}
This implies that 
\begin{align*}
G_n &\subset \left\{ m \in \mathbb N \mid \|v_m \varphi - \varphi v_m\| < \frac{1}{n + 1} \right \} \\
G_n &\subset \bigcap_{0 \leq k \leq n, b \in X_k^{\leq k}} \left\{ m \in \mathbb N \mid \|v_m \, b_m^1 \, v_m^*- b_m^2 \|^\sharp_\varphi < \frac{1}{n + 1} \right\}.
\end{align*} 
Since $G_n \in \u$, it follows that 
\begin{align*}
\left \{ m \in \mathbb N \mid \|v_m \varphi - \varphi v_m\| < \frac{1}{n + 1} \right\} & \in \u  \\
\forall 0 \leq k \leq n, \forall b \in X_k^{\leq k}, \quad  \left\{m \in \mathbb N \mid  \|v_m \, b_m^1 \, v_m^*- b_m^2 \|^\sharp_\varphi < \frac{1}{n + 1} \right\} & \in \u.
\end{align*} 
Since this holds for every $n \in \mathbb N$, we obtain $ \lim_{m \to \u} \|v_m \varphi - \varphi v_m\| = 0$. This implies that $(v_m)_m \in \mathfrak M^\u(M)$ and $u = (v_m)^\u \in \mathcal U((M^\u)_{\varphi^\u})$. This further implies that for every $k \in \mathbb N$ and every $b \in X_k^{\leq k}$, we have $\|u \pi_1(b)u^* - \pi_2(b)\|_{\varphi^\u}^\sharp = \lim_{m \to \u} \|v_m \, b_m^1 \, v_m^* - b_m^2\|_\varphi^\sharp = 0$. Since  $\bigcup_{k \in \mathbb N} X_k^{\leq k} = \bigcup_{n \in \mathbb N} X_n$, since for every $n \in \mathbb N$, the set $X_n$ is $\|\cdot\|_{\psi_n}^\sharp$-dense in $P_n$, since $\bigvee_{n \in \mathbb N} P_n = P$ and since for every $i \in \{1, 2\}$, $\pi_i : P \hookrightarrow M^\u$ is a normal embedding, this finally implies that $\pi_2(x) = u \pi_1(x) u^*$ for every $x \in P$.
\end{proof}

\begin{defn}\label{defn:AFD}
We say that a W$^*$-probability space $(P, \psi)$ is \emph{approximately finite dimensional} (AFD) if there exists an increasing sequence of finite dimensional W$^*$-probability subspaces $(P_n, \psi_n) \subseteq (P, \psi)$ such that $\bigvee_{n \in \mathbb N} (P_n, \psi_n) = (P, \psi)$.
\end{defn}

If $(P, \psi)$ is an AFD W$^*$-probability space, then $\psi$ is necessarily an almost periodic state on $P$.

\begin{examples}\label{ex:AFD}
Here are the main examples of AFD W$^*$-probability spaces:
\begin{enumerate}
\item Every AFD tracial von Neumann algebra $(M, \tau)$ endowed with a tracial faithful normal state is an AFD W$^*$-probability space.
\item For every $\lambda \in (0, 1)$, endow the type III$_\lambda$ Powers factor $R_\lambda$ with its canonical $\frac{2\pi}{|\log(\lambda)|}$-periodic faithful normal state $\varphi_\lambda$. Then $(R_\lambda, \varphi_\lambda)$ is an AFD W$^*$-probability space.
\item Endow the type III$_1$ Araki-Woods factor $R_\infty = R_{\lambda_1} \otimes R_{\lambda_2}$, where $\frac{\log(\lambda_1)}{\log(\lambda_2)} \notin \mathbb Q$, with the faithful normal state $\varphi = \varphi_{\lambda_1} \otimes \varphi_{\lambda_2}$. Then $(R_\infty, \varphi)$ is an AFD W$^*$-probability space.
\item For every AFD type III$_0$ factor $P$, there exists a faithful normal state $\varphi$ on $M$ such that $(M, \varphi)$ is an AFD W$^*$-probability space (see \cite[Theorem 1]{Co74b}).
\end{enumerate}
\end{examples}

We deduce that any AFD W$^*$-probability space has a unique embedding into $(M, \varphi)^\u$ up to unitary conjugacy.

\begin{thm}\label{thm:uniqueness}
Let $(P, \psi)$ be any {\em AFD} W$^*$-probability space. For every $i \in \{1, 2\}$, let $\pi_i : (P, \psi) \hookrightarrow (M, \varphi)^\u$ be any embedding. Then there exists $u \in \mathcal U((M^\u)_{\varphi^{\mathcal U}})$ such that $\pi_2 = \Ad(u) \circ \pi_1$.
\end{thm}

\begin{proof}
Let $(P_n, \psi_n) \subseteq (P, \psi)$ be an increasing sequence of finite dimensional W$^*$-probability subspaces such that $\bigvee_{n \in \mathbb N} (P_n, \psi_n) = (P, \psi)$. For every $n \in \mathbb N$, Lemma \ref{lem:atomic} implies that there exists $u_n \in \mathcal U((M^\u)_{\varphi^{\mathcal U}})$ such that $\pi_2(x) = u_n \pi_1(x) u_n^*$ for every $x \in P_n$. Then Lemma \ref{lem:increasing} implies that there exists $u \in \mathcal U((M^\u)_{\varphi^{\mathcal U}})$ such that $\pi_2(x) = u \pi_1(x) u^*$ for every $x \in P$.
\end{proof}

As a straightforward consequence of Theorem \ref{thm:uniqueness}, we obtain the following unique embedding property for the Powers factors $R_\lambda$ where $\lambda \in (0, 1)$.

\begin{cor}\label{cor:uniqueness-lambda}
Let $\lambda \in (0, 1)$ and $\pi : (R_\lambda, \varphi_\lambda) \hookrightarrow (R_\lambda, \varphi_\lambda)^{\u}$ be any embedding. Then there exists $u \in \mathcal U((R_\lambda^\u)_{\varphi_\lambda^{\mathcal U}})$ such that $\Ad(u) \circ \pi : (R_\lambda, \varphi_\lambda) \hookrightarrow (R_\lambda, \varphi_\lambda)^{\u}$ is the diagonal embedding.
\end{cor}

Combining Theorem \ref{thm:uniqueness} with Connes-St\o rmer transitivity theorem, we obtain the following unique embedding property for the Araki-Woods factor $R_\infty$.

\begin{cor}\label{cor:uniqueness-1}
Let $\psi$ be any faithful normal state on $R_\infty$ and $\pi : (R_\infty, \psi) \hookrightarrow (R_\infty, \psi)^{\u}$ any embedding. Then there exists $u \in \mathcal U((R_\infty^\u)_{\psi^\u})$ such that $\Ad(u) \circ \pi : (R_\infty, \psi) \hookrightarrow (R_\infty, \psi)^{\u}$ is the diagonal embedding.
\end{cor}

\begin{proof}
Denote by $E : R_\infty^\u \to \pi(R_\infty)$ the unique faithful normal conditional expectation such that $\psi \circ \pi^{-1} \circ E = \psi^\u$. Choose a faithful normal state $\varphi$ on $R_\infty$ such that $(R_\infty, \varphi)$ is an AFD W$^*$-probability space (see Example \ref{ex:AFD}(3)). Set $\phi = \varphi \circ \pi^{-1} \circ E \in (R_\infty^\u)_\ast$. By \cite[Theorem 4.20]{AH12}, there exists $w \in \mathcal U(R_\infty^\mathcal U)$ such that $ \phi  = \varphi^\u \circ \Ad(w)$.  (This result indeed holds for countably incomplete ultrafilters on arbitrary index sets.) Then $\pi_w = \Ad(w) \circ \pi : R_\infty \hookrightarrow R_\infty^{\u}$ is an embedding such that $\varphi^\u \circ \pi_w = \varphi^\u \circ \Ad(w) \circ \pi = \phi \circ \pi = \varphi$. Moreover, the faithful normal conditional expectation $E_w = \Ad(w) \circ E \circ \Ad(w^*) : R_\infty^\u \to w \pi(R_\infty)w^*$ satisfies $\varphi \circ \pi_w^{-1} \circ E_w = \varphi \circ \pi^{-1} \circ \Ad(w^*) \circ \Ad(w) \circ E \circ \Ad(w^*) = \varphi^\u$. Thus, $\pi_w : (R_\infty, \varphi) \hookrightarrow (R_\infty, \varphi)^\u$ is an embedding of W$^*$-probability spaces. 

By Theorem \ref{thm:uniqueness}, there exists $v \in \mathcal U((R_\infty^\u)_{\varphi^\u})$ such that $\Ad(v) \circ \pi_w : (R_\infty, \varphi) \hookrightarrow (R_\infty, \varphi)^\u$ is the diagonal embedding $\iota : (R_\infty, \varphi) \hookrightarrow (R_\infty, \varphi)^\u$. Set $u  = vw \in \mathcal U(R_\infty^\u)$. Then we have $\Ad(u) \circ \pi = \iota$ and $\varphi^\u \circ \Ad(u) \circ \pi = \varphi = \varphi^\u \circ \iota$. Denote by $E_\u : R_\infty^\u \to \iota(R_\infty)$ the canonical faithful normal expectation. Note that $E_u = \Ad(u) \circ E \circ \Ad(u^*) : R_\infty^\u \to u\pi(R_\infty)u^*$ is another faithful normal conditional expectation onto $\iota(R_\infty) = u\pi(R_\infty)u^*$. Since 
\begin{align*}
\varphi \circ \pi^{-1} \circ \Ad(u^*) \circ E_u &= \varphi \circ \pi^{-1} \circ \Ad(u^*) \circ \Ad(u) \circ E \circ \Ad(u^*) \\
&= \phi \circ \Ad(w^*) \circ \Ad(v^*) \\
&= \varphi^\u \\
&= \varphi \circ \iota^{-1} \circ E_\u,
\end{align*}
we have $E_\u = E_u =  \Ad(u) \circ E \circ \Ad(u^*)$. This further implies that 
\begin{align*}
\psi^\u &= \psi \circ \iota^{-1} \circ E_\u \\
&= \psi \circ \pi^{-1} \circ \Ad(u^*) \circ \Ad(u) \circ E \circ \Ad(u^*) \\
&= \psi \circ \pi^{-1} \circ E \circ \Ad(u^*) \\
&= \psi^\u \circ \Ad(u^*)
\end{align*}
and so $u \in \mathcal U((R_\infty^\u)_{\psi^\u})$, finishing the proof.
\end{proof}

We may also apply Theorem \ref{thm:uniqueness} to the structure of ultraproduct type III$_1$ factors. The next result extends \cite[Theorem 4.5]{FGL05}.

\begin{prop}\label{prop:prime}
Let $M$ be any $\sigma$-finite type ${\rm III}_1$ factor. Then $M^\u$ is a prime factor, that is, $M^\u \ncong M_1 \otimes M_2$ for any diffuse factors $M_1, M_2$.
\end{prop}

\begin{proof}
By contradiction, assume that $M^\u$ is not prime and write $M^\u = M_1 \otimes M_2$. For every $i = 1, 2$, choose a separable diffuse abelian von Neumann subalgebra $A_i \subset M_i$ with expectation. By Theorem \ref{thm:uniqueness}, there exists a unitary $w \in \mathcal U(M^\u)$ such that $w A_1 w^* = A_2$. 

For every $i = 1, 2$, choose a faithful normal state $\varphi_i$ on $M_i$ such that $A_i \subset (M_i)_{\varphi_i}$ and define the faithful normal conditional expectation $E_2 : M_1 \otimes M_2 \to M_2$ by $E_2 = \varphi_1 \otimes \id_{M_2}$. Choose a sequence of unitaries $(u_n)$ in $\mathcal U(A_1)$ such that $u_n \to 0$ weakly. For all $i \in \{1, 2\}$ and all $x_i,y_i \in M_i$, we have
$$\lim_n E_{A_2}((x_1 \otimes x_2)(u_n \otimes 1)(y_1 \otimes y_2)) = \lim_n \varphi_1(x_1 u_n y_1) x_2y_2 = 0.$$
By strong density of linear combinations of elementary tensors in $M_1 \otimes M_2$ and since $u_n \in A_1 \subset (M_1)_{\varphi_1}$, it follows that for all $x, y \in M_1 \otimes M_2 = M^\u$, we have
$$\lim_n E_{A_2}(x(u_n \otimes 1)y) = 0.$$
Applying the above result to $x = w$ and $y = w^*$ and since $w A_1 w^* = A_2$, we obtain 
$$1 = \lim_n \|w(u_n \otimes 1)w^*\|_{\varphi_2} = \lim_n \|E_{A_2}(w(u_n \otimes 1)w^*)\|_{\varphi_2} = 0.$$
This is a contradiction.
\end{proof}

\section{Keisler's Sandwich Theorem}

In this section we prove Keisler's Sandwich Theorem (Theorem \ref{keisler}).  To simplify the matter, we work in the traditional $[0,1]$-valued version of continuous logic presented in \cite{BBHU08}.  We freely use the notation and terminology established in \cite{BBHU08}.

Fix a language $L$.  Let $L_\exists$ be the language obtained by adding a predicate $P_\varphi$ for every existential $L$-formula $\varphi$.  It is clear that an embedding between $L_\exists$-structures is an existential embedding of their $L$-reducts and, conversely, any existential embedding between $L$-structures is an embedding of their canonical expansions to $L_\exists$-structures.

\begin{lem}
Any restricted quantifier-free $L_\exists$-formula is equivalent to both an $\forall_2$ $L$-formula and a $\exists_2$ $L$-formula.
\end{lem}

\begin{proof}
We prove the lemma by induction on the complexity of formulae.  The main case is the connective $\dminus$.  This follows from the following calculations:
\begin{itemize}
    \item $(\sup_x \inf_y \varphi)\dminus (\inf_z\sup_w\psi)\equiv \sup_x\sup_z\inf_y\inf_w (\varphi\dminus \psi)$ and
    \item $(\inf_x\sup_y \varphi)\dminus (\sup_z\inf_w\psi)\equiv \inf_x\inf_z\sup_y\sup_w (\varphi\dminus \psi).$\qedhere
\end{itemize}
\end{proof}

In what follows, given an $L$-structure $N$, $\Th_{\forall_n}(N)$ denotes the closed conditions of the form $\sigma=0$, where $\sigma$ is a $\forall_n$-sentence for which $\sigma^N=0$.  Similarly, if $T$ is a theory, we let $T_{\forall_n}$ denote the collection of closed conditions $\sigma=0$, where $\sigma$ is a $\forall_n$-sentence for which $T\models \sigma=0$.  The corresponding notions with $\forall_n$ replaced by $\exists_n$ are defined analogously.  It is routine to verify that $M\models \Th_{\forall_n}(N)$ if and only if $N\models \Th_{\exists_n}(M)$.

\begin{thm}\label{keisler1}
For $L$-structures $M$ and $N$, the following are equivalent:
\begin{enumerate}
    \item $M\models \Th_{\forall_n}(N)$;.
    \item There is an elementary extension $N'$ of $N$ for which there is an $(M,N',n)$-ultrapower sandwich.
\end{enumerate}
\end{thm}

\begin{proof}
We prove this by induction on $n$.  The case $n=1$ is well-known.

Now suppose that $n>1$.  If $S$ is an $n$-ultrapower sandwich, let $S'$ denote the sandwich with the last element removed.

First suppose that there is an $(M,N',n)$-ultrapower sandwich $S$ with $N'$ an elementary extension of $N$.  Then $S$ is an $(M,N',n-1)$-ultrapower sandwich with respect to $L_\exists$.  By induction, we have that $M$ is a model of $\Th_{\forall_{n-1}}(N)$ with respect to the language $L_\exists$.  It is clear that $\Th_{\forall_n}(N)$ with respect to $L$ is contained in $\Th_{\forall_{n-1}}(N)$ with respect to $L_\exists$, so $M\models \Th_{\forall_n}(N)$, as desired.

Now suppose that $M\models \Th_{\forall_n}(N)$.  By the previous lemma, we have that $M\models \Th_{\forall_{n-1}}(N)$ with respect to $L_\exists$.  By induction, there is an elementary extension $N'$ of $N$ for which there is an $(M,N',n-1)$-ultrapower sandwich with respect to $L_\exists$, that is, the embeddings are all existential embeddings.  In particular, the chain can be extended by one more element if the embeddings are not required to be existential.  Thus, there is an $(M,N',n)$-ultrapower sandwich, as desired.
\end{proof}

We can now prove the above promised sandwich theorems:

\begin{cor}
For any $L$-theory $T$, the following are equivalent:
\begin{enumerate}
    \item $T$ is $\forall_n$-axiomatizable.
    \item Whenever there is an $(M,N,n)$-ultrapower sandwich with $N\models T$, we also have that $M\models T$.
\end{enumerate}
\end{cor}

\begin{proof}
First suppose that $T$ is $\forall_n$-axiomatizable and consider an $(M,N,n)$-ultrapower sandwich with $N\models T$.  Then by Theorem \ref{keisler1}, we have that $M\models \Th_{\forall_n}(N)$.  Since $N\models T$, we have that $T_{\forall_n}\subseteq \Th_{\forall_n}(N)$.  It follows that $M\models T_{\forall_n}$.  Since $T$ is $\forall_n$-axiomatizable, we have that $M\models T$, as desired.

Conversely, suppose that (2) holds and let $M\models T_{\forall_n}$.  We wish to show that $M\models T$.  Consider the set
$$\Sigma:=T\cup \{\sigma\geq \frac{\epsilon}{2} \ : \ \sigma^M\geq \epsilon, \ \sigma \text{ is a }\forall_n\text{-sentence}\}.$$

If $\Sigma$ were unsatisfiable, then there would be $\sigma_1,\ldots,\sigma_m$ and $\epsilon$ such that $\sigma_i^M\geq \epsilon$ for all $i$ and yet $T\models \max_{1\leq i\leq m}(\sigma_i\dminus \frac{\epsilon}{2})=0$.  Since this latter sentence is still $\forall_n$, we have that it belongs to $T_{\forall_n}$, contradicting the fact that $M\models T_{\forall_n}$.

Let $N\models \Sigma$.  Note then that $N\models T$.  Moreover, $M\models \Th_{\forall_n}(N)$.  Indeed, if $\sigma$ is a $\forall_n$ sentence such that $\sigma^N=0$, then $\sigma^M=0$, else there is $\epsilon>0$ such that $\sigma^M\geq \epsilon$, whence $\sigma^N\geq \frac{\epsilon}{2}$, a contradiction.

By Theorem \ref{keisler1}, there is an elementary extension $N'$ of $N$ for which there is an $(M,N',n)$-ultrapower sandwich.  By (2) and the fact that $N'\models T$, we have that $M\models T$, as desired.
\end{proof}

\begin{cor}
For any $L$-theory $T$, the following are equivalent:
\begin{enumerate}
    \item $T$ is $\exists_n$-axiomatizable.
    \item Whenever there is an $(M,N,n)$-ultrapower sandwich with $M\models T$, we also have that $N\models T$.
\end{enumerate}
\end{cor}

\begin{proof}
First suppose that $T$ is $\exists_n$-axiomatizable and consider an $(M,N,n)$-ultrapower sandwich with $M\models T$.  By Theorem \ref{keisler1}, we have that $M\models \Th_{\forall_n}(N)$.  Note then that $N\models \Th_{\exists_n}(M)$.  Indeed, if $\sigma$ is an $\exists_n$-sentence and $\sigma^M=0$, then $\sigma^N=0$, else $\sigma^N\geq \epsilon$.  Since $\epsilon\dminus \sigma$ is equivalent to a $\forall_n$-sentence, we get that $\epsilon\dminus \sigma\in \Th_{\forall_n}(N)$, whence $(\epsilon\dminus \sigma)^M=0$, contradicting that $\sigma^M=0$.  Since $T$ is $\exists_n$-axiomatizable, we have that $N\models T$, as desired.

Now suppose that (2) holds and $M\models T_{\exists_n}$.  We wish to show that $M\models T$.  The exact same argument as in the previous theorem shows that there is $N\models T$ such that $\sigma^N\geq \epsilon$ whenever $\sigma$ is a $\exists_n$-sentence with $\sigma^M\geq \epsilon$.  It follows that $M\models \Th_{\exists_n}(N)$.  Arguing as above, we have that $N\models \Th_{\forall_n}(M)$, so there is an $(N,M',n)$-ultrapower sandwich with $M'$ an elementary extension of $M$.  By (2) and the fact that $N\models T$, we have that $M'\models T$, whence $M\models T$, as desired.  
\end{proof}

\bibliographystyle{plain}

\end{document}